\pdfoutput=1
\documentclass[english]{jnsao}
\usepackage[utf8]{inputenc}

\usepackage{amsfonts}
\usepackage{amsmath,amsthm,amscd,mathrsfs,setspace}
\usepackage{booktabs}
\usepackage{mathtools}
\usepackage{cite}
\usepackage{latexsym,epsf,epsfig}
\usepackage{color}
\usepackage[left=1in,right=1in,top=1in,bottom=1in]{geometry}
\usepackage[mathscr]{euscript}
\usepackage{dutchcal}
\usepackage{algorithm}
\usepackage{algorithmicx}
\usepackage{algpseudocode}
\usepackage[scientific-notation=true]{siunitx}
\usepackage{subcaption}
\usepackage{tikz}
\usepackage{pgfplots}
\pgfplotsset{compat=1.16}

\newcommand{\mb}{\mathbf}
\newcommand{\grad}{\nabla}
\newcommand{\ds}{\displaystyle}
\newcommand{\N}{\mathbb{N}}
\newcommand{\R}{\mathbb{R}}
\newcommand{\Z}{\mathbb{Z}}
\newcommand{\weakstarto}{\stackrel{\ast}{\rightharpoonup}}

\newcommand*\dd{\mathop{}\!\mathrm{d}}
\DeclareMathOperator*{\TV}{TV}
\DeclareMathOperator*{\BV}{BV}
\DeclareMathOperator*{\BVW}{BV_W}

\DeclareMathOperator*{\TR}{TR}

\theoremstyle{plain}
\newtheorem{assumption}[theorem]{Assumption}

\theoremstyle{remark}
\numberwithin{equation}{section}
\numberwithin{theorem}{section}
\numberwithin{remark}{section}
\numberwithin{condition}{section}

\manuscriptsubmitted{2022-12-21}
  \manuscriptaccepted{2024-04-09}
  \manuscriptvolume{5}
  \manuscriptnumber{10529}
  \manuscriptyear{2024}
  \manuscriptdoi{10.46298/jnsao-2024-10529}

\usepackage[nameinlink,capitalize]{cleveref}

\title{Input regularization for integer
optimal control in BV with applications to
control of poroelastic and poroviscoelastic systems}
\author{Lorena Bociu%
    \thanks{Department of Mathematics, North Carolina State University, Raleigh, NC, USA (\email{lvbociu@ncsu.edu})}
    \and
    Paul Manns\thanks{TU Dortmund, Dortmund, Germany (\email{paul.manns@tu-dortmund.de})}
    \and
    Marvin Severitt\thanks{Faculty of Mathematics, TU Dortmund,
    	Dortmund, Germany (\email{marvin.severitt@tu-dortmund.de})}
    \and
    Sarah Strikwerda\thanks{Department of Mathematics, North Carolina State University, Raleigh, NC, USA (\email{slstrikw@ncsu.edu})}
\author{Paul Manns%
\thanks{\email{paul.manns@tu-dortmund.de}}}}
\shortauthor{Bociu, Manns, Severitt, Strikwerda}

\acknowledgements{
Funding: Lorena Bociu was partially supported by NSF CAREER 1555062. Paul Manns and Marvin Severitt acknowledge funding by Deutsche Forschungsgemeinschaft (DFG) under project no. 515118017.
}

\crefname{theorem}{Theorem}{Theorems}
\Crefname{theorem}{Theorem}{Theorems}
\crefname{assumption}{Assumption}{Assumptions}
\Crefname{assumption}{Assumption}{Assumptions}
\crefname{lemma}{Lemma}{Lemmas}
\Crefname{lemma}{Lemma}{Lemmas}
\crefname{definition}{Definition}{Definitions}
\Crefname{definition}{Definition}{Definitions}
\crefname{proposition}{Proposition}{Propositions}
\Crefname{proposition}{Proposition}{Propositions}
\crefname{algorithm}{Algorithm}{Algorithms}
\Crefname{algorithm}{Algorithm}{Algorithms}

\usepackage{marginnote}

\begin{document}
%%%%%%%%%%%%%%%%%%%%%%%%%%%%%%%%%%%%%%%%%%%%%%%

\maketitle

\begin{abstract}
We revisit a class of integer optimal control problems for which a trust-region method has been proposed and analyzed in \cite{leyffer2022sequential}. While the algorithm proposed in \cite{leyffer2022sequential} successfully solves the class of optimization problems under consideration, its convergence analysis requires restrictive regularity assumptions. There are many examples of integer optimal control problems involving  partial differential equations where these regularity assumptions are not satisfied. In this article we provide a way to bypass the restrictive regularity assumptions by introducing an additional partial regularization
    of the control inputs by means of mollification and proving a $\Gamma$-convergence-type
    result when the support parameter of the mollification
    is driven to zero. We highlight the applicability of this theory in the case of fluid flows through deformable porous media equations that arise
    in biomechanics. We show that the regularity assumptions are violated in the case of poroviscoelastic systems, and thus one needs to use the regularization of the control input introduced in this article. Associated 
numerical results show that while the homotopy can help to find better
objective values and points of lower instationarity, the practical
performance of the algorithm without the input regularization 
may be on par with the homotopy.
\end{abstract}

%%%%%%%%%%%%%%%%%%%%%%%%%%%%%%%%%%%%%%%%%%%%%%%
\section{Introduction}

We are interested in solving the following 
optimization problem
\begin{gather}\label{eq:p}
\begin{aligned}
    \min_{w \in \BV(0,T)}\ & j(w) + \alpha \TV(w)\\
    \text{s.t.}\quad\quad & w(t) \in W \subset \Z
    \text{ for almost all (a.a.) } t \in (0,T),
\end{aligned}\tag{P}
\end{gather}
where parameter $\alpha > 0$, time $T > 0$,  $W \subset \Z$ is a finite set of integers, and $\TV$ denotes the
total variation of the $W$-valued control input function $w$. The functional
$j : L^2(0,T) \to \R$ is the objective that takes the
form $j(w) \coloneqq J(Gw,w)$, where $J : X\to \R$
is a coercive and lower semicontinuous function on
a Banach space $X$ that is the state space of some partial differential equation (PDE).
The function $G : L^2(0,T) \to X$ is the continuous
solution operator of a PDE, which in this article will be the solution operator
of a coupled system describing fluid flow through deformable porous media (see \Cref{sec:biots_model}). We note that \eqref{eq:p}
admits a solution in this setting, see \cite{leyffer2001integrating}.

The optimization problem \eqref{eq:p} falls in the class of so-called
\emph{integer optimal control problems}, which allow to model non-smooth behavior
	by restricting to discrete changes in distributed control variables.
Driven by versatile applications from the optimization of supply
and traffic networks 
\cite{martin2006mixed,goettlich2017partial,hante2017challenges,gottlich2019partial}
over automotive control \cite{kirches2013mixed,gerdts2005solving}
to topology optimization \cite{sigmund2013topology,haslinger2015topology}, 
this problem class has attracted considerable research interest in
recent years.
Different methods have been proposed to treat integer optimal control 
problems. One of them is the combinatorial integral approximation decomposition \cite{sager2011combinatorial} that splits the
optimization into the solution of a relaxed problem, where $W$ is 
replaced by a \emph{one-hot encoding} and then the convex hull is analyzed
and a fast algorithm computes a $W$-valued control from the
relaxation \cite{sager2005numerical,sager2012integer,hante2013relaxation,manns2020multidimensional}. This method requires the ability to produce highly oscillating
control functions, which are undesirable in many applications
and can therefore not be applied to \eqref{eq:p} if $\alpha > 0$.

The $\TV$-term that influences \eqref{eq:p} for $\alpha > 0$ has been
prevalent in mathematical image analysis since the 1990s, see in particular
the work \cite{rudin1992nonlinear}. We give the references
\cite{vogel1996iterative,chambolle1997image,chan1998total,fornasier2009subspace,bredies2010total,lellmann2014imaging,hintermuller2017optimal} but note that they reflect only a small portion of the research
in this area.  Several authors have incorporated $\TV$-terms in optimal control
problems \cite{CKK, CKL, loxton2012control,kaya2020optimal,engel2021optimal},
in tight relaxations of integer optimal control problems
from topology optimization \cite{clason2018total}, and in
the approximation step of the combinatorial 
integral approximation decomposition
\cite{bestehorn2019switching,sager2021mixed,bestehorn2021mixed}.

A trust-region method has been proposed and analyzed to directly solve problems of the form 
\eqref{eq:p} in \cite{leyffer2022sequential}. Therein, the non-smoothness of
the subproblems that arises from the distributed integer variables is handled explicitly
so that the trust-region subproblems are integer linear
programs after discretization. They can be solved efficiently
with graph-based \cite{severitt2022efficient} and dynamic
programming-based approaches \cite{marko2022integer}.
For one-dimensional time domains $(0,T)$, like the one in \eqref{eq:p},
the $\TV$-term of a $W$-valued function $w$ is the sum of the jump
heights of the function $w$, implying that only finitely many
jumps occur because the height of a single jump is bounded below
(by 1 in the case  of $W \subset \Z$; always by some constant
if $|W| < \infty$). This constitutes a desirable regularization because the application 
underlying \eqref{eq:p} usually does not permit infinitely many jumps
or general high-frequency  switching between different control modes.
This is also the case for the class of PDEs that we consider in this article to
constrain our problem \eqref{eq:p} when considering the application of tissue engineering.
This is due to the fact that in the laboratory, only a finite set of values of the controls,
e.g.\ loads that are applied in confined compression testing, are used. We highlight that although we restrict ourselves to integer-valued
controls, all of the theory can be transferred straightforwardly to the case that $W$ is a finite subset of $\R$
because the $\TV$-seminorm difference between two controls is still bounded below by a positive constant if it is
not zero. Therefore, we believe that these equations constitute good test cases for our algorithm.

While these optimization problems seem to be suited for the algorithmic framework proposed in \cite{leyffer2022sequential}, the regularity
assumptions for its convergence analysis cannot always be satisfied. Therefore, in this article, we advance the algorithmic methodology
by enforcing the necessary regularity by adding a regularization of the control input when passed to $j$. Specifically, the function $j$
is altered to $j \circ K_\varepsilon$, where $K_\varepsilon$ is a
convolution operator arising from a standard mollification with parameter
$\varepsilon$. 
In this article we make the following contributions
in light of the algorithmic framework proposed in
\cite{leyffer2022sequential}.

\paragraph{Contributions} First we verify the regularity assumptions required
in \cite{leyffer2022sequential} for the altered objective. We prove
the lower and upper bound inequalities ($\Gamma$-convergence) of the
altered optimization problems when driving $\varepsilon$
(the parameter controlling the
size of the support of the mollifier) to zero. The $\Gamma$-convergence
result is achieved with respect to weak$^*$ and
strict convergence in the weak$^*$ closed subset of functions of
bounded variations that are feasible for \eqref{eq:p}.

Consequently,  global minimizers of the altered optimization
problems converge to global minimizers  of \eqref{eq:p}. However, the lower and upper bound inequalities do not
imply that the same holds true for stationary points like the ones that are produced
by the algorithm proposed in \cite{leyffer2022sequential}. We
consider a homotopy
that drives $\varepsilon \to 0$ and applies meaningful termination
criteria for each run (tightening of minimal trust-region radius
to determine that no progress is made for $\varepsilon \to 0$ and
achievement of a certain predicted reduction). 
We show that the cluster points
are strict limits of their approximating sequences, implying that
the homotopy does not \emph{overlook} cheap reductions of the objective
that can, for example, be obtained by removing small jumps from the control.

%As an application of our theory, we consider 
We highlight the applicability and benefits of our theory in the case of linear poroelastic and poroviscoelastic systems with incompressible constituents and distributed or boundary controls, with motivation coming from biomedicine. We show that the regularity assumptions required
in \cite{leyffer2022sequential}  are satisfied in the case of poroelastic systems. As a consequence, the theory and algorithm provided in \cite{leyffer2022sequential} can be applied.  In comparison, the regularity assumptions are violated in the case of poroviscoelastic systems. Therefore, for these systems, one needs to use the regularization of the control input introduced in this article. 

Lastly, we provide numerical results for two instances of the class of considered
PDEs that differ in their dynamics and analytical properties. The
numerical results show that while the homotopy can help to find better
objective values and points of lower instationarity, the practical
performance of the algorithm without the input regularization 
may be on par with the homotopy. Consequently, the lack of
regularity may not always impair the practical performance and may
therefore be outweighed
by the consumption of much less running time than the homotopy.

\paragraph{Structure of the paper.} In \Cref{sec:primer}, 
we introduce our notation, briefly recall functions of bounded variation, 
and define the
closed subset that corresponds to the feasible set of our optimization problem. In \Cref{sec:slip}, we 
describe the sequential linear integer programming (SLIP) algorithm provided in
\cite{leyffer2022sequential}. We
introduce and analyze the effect of the control input regularization in 
\Cref{sec:input_regularization}. The considered class of
fluid-solid mixture systems and the discussion of the regularity
assumptions of \Cref{alg:slip} with respect to these
coupled systems of  PDEs are given in \Cref{sec:biots_model}.
We provide our computational setup, experiments, and results
in \Cref{sec:comp}. Finally, we draw our  conclusions 
in \Cref{sec:conclusion}.

\section{Notation and Primer on Functions of Bounded Variation}\label{sec:primer}

\paragraph{Notation.} Let $X$ be a Banach space. As usual,  we denote its topological
dual space by the symbol $X^*$. For a given bounded, Lipschitz domain $\Omega$,  $L^2(\Omega)$ is the Hilbert space of square-integrable functions with inner product given by $(\cdot, \cdot)$. When the domain $\Omega$ is not clear from context, the $L^2$ inner product will be denoted as $(\cdot, \cdot)_{\Omega}$. Furthermore, we use the standard notation $\ds H^1_{\Gamma_*}(\Omega) = \left\{f \in H^1(\Omega)\,\middle|\, \tau f\big|_{\Gamma_*} = 0 \right\}$ where $\tau$ is the trace operator, for any $\Gamma_* \subseteq \partial \Omega$.
Additionally, for any Hilbert space $Y$, we define the space  $$L^2(0,T;Y)= \left\{u:[0,T] \to Y \  \middle| \  u\text{ is measurable and }\int_0^T \|u(t)\|_Y^2dt <\infty\right\}$$ with the inner-product $(u,v)_{L^2(0,T;Y)}=\int_0^T (u(t),v(t))_Y dt$. Similarly, $H^1(0,T;Y)$ is the set of functions in $L^2(0,T;Y)$ with a time derivative in the weak sense, $u_t: [0,T] \to Y$, that belongs to $L^2(0,T;Y)$. The inner product in $H^1(0,T;Y)$ is $(u,v)_{H^1(0,T;Y)}=\int_0^T (u,v)_Y + (u_t, v_t)_Y dt$. \\ 
\paragraph{Functions of Bounded Variation.}
We give a brief summary and state the properties of functions of bounded 
variation, which we require in the remainder of the paper. For a detailed introduction,
we refer the reader to the monograph \cite{ambrosio2000functions}. 
First, we recall that a function $f : (0,T) \to \R$ is defined to be
of bounded  variation or in the space $\BV(0,T)$ if $f \in L^1(0,T)$
and
\[ \TV(f) \coloneqq \sup
\left\{ \int_0^T f(t) \phi'(t)\dd t\,\Bigg|\,
\phi \in C^1_c(0,T)\ \text{ and }\ \sup_{\mathclap{s \in (0,T)}} |\phi(s)| \le 1
\right\} < \infty.
\]
We recall that a sequence $(w^n)_{n \in \mathbb{N}} \subset \BV(0,T)$ is said to
converge \emph{weakly-$^*$}
to a function $w \in \BV(0,T)$ if
$w^n \to w$ in $L^1(0,T)$ and $\limsup_{n\to\infty} \TV(w^n) < \infty$.
Moreover, $(w^n)_{n \in \mathbb{N}}$ is said to converge \emph{strictly} to $w$
if in addition $\TV(w^n) \to \TV(w)$.
We define the subset of $\BV(0,T)$ that corresponds to the
feasible set of the optimization problem \eqref{eq:p} as
\[ \BVW(0,T) = \{ w \in \BV(0,T) \,|\, w(t) \in W \text{\ for a.a.\ } t \in [0,T]\}. \]
It is important for our analysis that the subset $\BVW(0,T)$ 
is closed with respect to weak-$^*$ and strict
convergence in $\BV(0,T)$. The closedness follows from the
fact that sequences of $W$-valued functions that converge
in $L^1(0,T)$ also have $W$-valued limits, which is stated
explicitly for our context in \cite[Lemma 2.2]{leyffer2022sequential}.

The analysis of the algorithm in \cite{leyffer2022sequential}---the
starting point of our work---makes use of regularity
conditions, particularly continuity properties, that are defined for
input functions in $L^2(0,T)$. While the norm-topologies of
$L^1(0,T)$ and $L^2(0,T)$ are different, we note that
convergence of a sequence of $W$-valued functions in $L^1(0,T)$
implies  convergence in $L^2(0,T)$ as well (due to the fact that $W$ is finite, implying a uniform
$L^\infty(0,T)$-bound on any sequence of functions).
This can be seen as follows. Let $(u_n)_{n \in \mathbb{N}} \subset \BVW(0,T)$ be a sequence such that $u_n \weakstarto u$ in $\BVW(0,T)$. Let $w_{max}$ be the maximum value of $W$. We have 
\begin{align*}\|u_n-u\|^p_{L^p(0,T)} &= \int_0^T |u_n(t)-u(t)|^p dt \leq \int_0^T 2^{p-1}(|u_n(t)|^{p-1} + |u(t)|^{p-1})|u_n(t) -u(t)| \ dt\\
& \leq C_{p, w_{max}}\|u_n(t) -u(t)\|_{L^1(0,T)}\end{align*}
Consequently, we will frequently use that sequences of functions that converge 
weakly-$^*$ or strictly in $\BVW(0,T)$ also converge in $L^2(0,T)$.

Finally, we recall Young's inequality for convolution, see \cite{castillo2016} p. 319,  since it is used several times in the following sections: Given $f \in L^p(0,T)$ and $g\in L^q(0,T)$ such that $\frac{1}{p} + \frac{1}{q}=1+ \frac{1}{r}$ with $1\leq p,q, r \leq \infty$,
   \begin{equation} \label{Yineq} \|f*g\|_{L^r(0,T)} \leq C\|f\|_{L^p(0,T)}\|g\|_{L^q(0,T)}.\end{equation}
\section{Sequential Linear Integer Programming Algorithm}\label{sec:slip}
In order to provide a self-contained article, we provide the SLIP algorithm, which is a function space
algorithm to solve \eqref{eq:p} to stationarity \cite{leyffer2022sequential}. Conceptually, the SLIP
algorithm is a trust-region method that solves a sequence of trust-region subproblems.
We briefly introduce the trust-region problem below before laying out the algorithm.

The trust-region subproblem takes a feasible control $w$ for \eqref{eq:p} and a
trust-region radius $\Delta > 0$ as inputs and reads
\begin{gather}\label{eq:tr}
\begin{aligned}
\min_{d \in L^2(0,T)} & (\nabla j(w), d)_{L^2(0,T)} + \alpha \TV(w + d) - \alpha\TV(w)
\eqqcolon \ell(w, d) \\
\text{s.t.}\quad
& w(s) + d(s) \in W \text{ for a.a.\ } s \in (0,T),\\
& \|d\|_{L^1(0,T)} \le \Delta,
\end{aligned}
\tag{TR}
\end{gather}
where we assume that $j: L^2(0,T) \to \mathbb{R}$ is Fr\'{e}chet differentiable, in particular $\nabla j(w) \in L^1(0,T).$ In section \ref{sec:input_regularization} we will discuss the further assumptions made in Assumption 4.1 in \cite{leyffer2022sequential} which are required for the convergence analysis.
An instance of the problem class \eqref{eq:tr} for given $w$ and $\Delta > 0$ is denoted by $\TR(w,\Delta)$ in the remainder.

We state the trust-region algorithm that solves subproblems of the form \eqref{eq:tr}
in \Cref{alg:slip} \cite{leyffer2022sequential}. The algorithm consists of two nested loops. In every iteration of the outer loop,
which is indexed by $n \in \N$, the trust-region radius is reset to the input $\Delta^0 > 0$. Then the inner loop, which is indexed
by $k \in \N$, is executed. In each inner iteration, a trust-region subproblem $\TR(w^{n-1},\Delta^{n,k})$
is solved for the current trust-region radius, $\Delta^{n,k}$, and the previously accepted iterate $w^{n - 1}$ or the input $w^0$ (if $n - 1 = 0$).
We highlight that the trust-region subproblems become integer linear programs after discretization, see \cite{leyffer2022sequential}, which can be solved to optimality with a pseudo-polynomial algorithm as detailed in \cite{severitt2022efficient,marko2022integer}. This is a deviation from the standard literature, where the convergence theory is developed using a Cauchy point which only guarantees a sufficient decrease and not optimality for the trust-region subproblem, see for example chapter 12 in \cite{Conn2000TRMethods}.
If the predicted reduction (measured as the negative objective value of the trust-region subproblem) is zero, then the algorithm terminates.
If the solution of the trust-region subproblem is acceptable (the ratio of actual reduction and predicted reduction
is larger than the input $\sigma > 0$), then the inner loop terminates with new iterate $w^n$.
If the step is rejected, then the trust region is reduced and another iteration of the inner loop is executed.
\begin{algorithm}[H]
\caption{Trust-region algorithm from \cite{leyffer2022sequential}}\label{alg:slip}
\textbf{Input: } Initial guess  $w^0$ (feasible for \eqref{eq:p}), $\Delta^0 > 0$, $\sigma \in (0,1)$
\begin{algorithmic}[1]
\For{$n = 1,\ldots$}
\State $k \gets 0$
\State $\Delta^{n,0} \gets \Delta^0$
\Repeat
\State $d^{n,k} \gets$ minimizer of $\TR(w^{n-1},\Delta^{n,k})$
\Comment{Solve trust-region subproblem.}
\If{$\ell(w^{n-1}, d^{n,k}) = 0$}
\Comment{Predicted reduction is zero $\Rightarrow$ terminate.}\label{ln:terminate}
\State Terminate with solution $w^{n-1}$.
\ElsIf{$\frac{j(w^{n - 1}) + \alpha \TV(w^{n-1}) - j(w^{n-1} + d^{n,k}) - \alpha \TV(w^{n-1} + d^{n,k})}{-\ell(w^{n-1},d^{n,k})} < \sigma $}\label{ln:suff_red}
\Comment{Reject step.}
\State $\Delta^{n,k+1} \gets \Delta^{n,k} / 2$
\Else
\Comment{Accept step.}
\State $w^n \gets w^{n-1} + d^{n,k}$
\EndIf
\State $k \gets k + 1$
\Until{$\frac{j(w^{n - 1}) + \alpha \TV(w^{n-1}) - j(w^{n-1} + d^{n,k - 1}) - \alpha \TV(w^{n-1} + d^{n,k - 1})}{-\ell(w^{n-1},d^{n,k-1})} \ge \sigma$}\label{ln:inner_loop_condition}
\EndFor
\end{algorithmic}
\end{algorithm}
The main known convergence result on \Cref{alg:slip} to this point is
that its iterates converge to so-called L-stationary points under a
suitable regularity assumption on $j$
\cite{leyffer2022sequential}. A feasible point $w$ is L-stationary
if the objective \eqref{eq:p} cannot be improved
further by perturbing the locations of its jumps on $(0,T)$. Such perturbations leave the term $\alpha \TV(w)$ unchanged
and only affect $j(w)$, yielding a condition on $\nabla j(w)$. In particular, the condition coincides with
$\nabla j(w)(t_i) = 0$ for all jump locations $t_i$ of $w$ if $\nabla j(w)$ is a continuous function.
The formal definition is given below.
\begin{definition}\label{dfn:l_stationary}
Let $v \in \BVW(0,T)$ with representation
$v = \chi_{(t_0,t_1)} a_1 + \sum_{i=1}^{N-1} \chi_{[t_i,t_{i+1})} a_{i+1}$
for $N \in \N$, $t_0 = 0$, $t_N = T$, $t_i < t_{i+1}$ for $i \in \{0,\ldots,N-1\}$,
and $a_i \in W$ for $i \in \{1,\ldots,N\}$, $a_i \neq a_{i+1}$,
be given. Let $j: L^2(0,T) \to \mathbb{R}$ be Fr\'{e}chet
differentiable, in particular $\nabla j(v) \in L^1(0,T)$.
Then $v$ is L-stationary for \eqref{eq:p} if
\begin{enumerate}
\setlength\itemsep{.5em}
\item $\overline{D^-_i}(\nabla j(v))
       \ge 0 \ge \underline{D_i^+}(\nabla j(v))$
       if $a_i < a_{i+1}$, and
\item $\underline{D^-_i}(\nabla j(v))
       \le 0 \le \overline{D_i^+}(\nabla j(v))$ if $a_{i+1} < a_i$,
\end{enumerate}
where
\[ \overline{D_{i}^-} (\nabla j(v)) \coloneqq \limsup_{h \downarrow 0} \frac{1}{h}\int_{t_i - h}^{t_i} \nabla j(v)(s)\dd s,\quad
   \underline{D_{i}^-} (\nabla j(v)) \coloneqq \liminf_{h \downarrow 0}  \frac{1}{h}\int_{t_i - h}^{t_i} \nabla j(v)(s)\dd s 
\]
and
\[ \underline{D_{i}^+} (\nabla j(v)) \coloneqq \liminf_{h \downarrow 0}	 \frac{1}{h}\int_{t_i}^{t_i+h} \nabla j(v)(s)\dd s,\quad
   \overline{D_{i}^+} (\nabla j(v)) \coloneqq \limsup_{h \downarrow 0} \frac{1}{h}\int_{t_i}^{t_i+h} \nabla j(v)(s)\dd s
\]
for $i \in \{1,\ldots,N-1\}$.
\end{definition}
Note that \Cref{dfn:l_stationary} is well posed because every function in $\BVW(0,T)$ can be written in the
claimed form, see, for example, \cite[Proposition 4.4]{leyffer2022sequential}.

\section{Input Regularization}\label{sec:input_regularization}
This section is structured as follows. First, we recall the regularity assumptions imposed on the Hessian of the reduced objective $j$ introduced in \cite{leyffer2022sequential}.  Under these assumptions, convergence to L-stationary points
of the iterates produced by \Cref{alg:slip} can be achieved. 
Secondly, motivated by our applications, we introduce weaker assumptions
on the Hessian's regularity and show that the required regularity 
assumptions for convergence of \Cref{alg:slip} can always be
verified by  regularizing (smoothing) the input of $j$ provided that
these weaker assumptions hold.
 Then we prove
$\Gamma$-convergence in the case when the regularization
is carried out by a positive mollifier and when driving the support parameter to 
zero. We also show that the limits of the final iterations of \Cref{alg:slip}
(under realistic termination criteria), which are in general not global
minimizers but only (approximately) L-stationary points, are strict.
\begin{assumption}[Assumption 4.1 in \cite{leyffer2022sequential}]
\label{ass:strong_assumption}
Let $j : L^2(0,T) \to \R$ be twice Fr\'{e}chet differentiable. Moreover, for all $w \in L^2(0,T)$
\[ |\nabla^2 j(w)(\psi,\phi)| \le C \|\psi\|_{L^1(0,T)}\|\phi\|_{L^1(0,T)} \]
holds for some $C>0$ and all $\psi,\phi \in L^2(0,T)$.
\end{assumption}
Under \cref{ass:strong_assumption} on the reduced objective
it can be shown that the iterates produced by \Cref{alg:slip}
converge to L-stationary points \cite{leyffer2022sequential}.
We remark here that in \cite{leyffer2022sequential} it is assumed that $j$ is not only twice Fr\'{e}chet
differentiable,  but that the second derivative $w \mapsto \nabla^2 j(w)$ is also continuous. This
assumption enters the proofs of Lemma 4.10 and Theorem 4.23 in \cite{leyffer2022sequential} 
indirectly through the employed variant of Taylor's theorem,
Proposition A.1 in \cite{leyffer2022sequential}, which
states that for all $u$, $v \in L^2(0,T)$ there exists $\xi$ in the
line segment between $u$ and $v$ such that
$j(v) = j(u) + (\nabla j(u), v - u)_{L^2(0,T)} +
\tfrac{1}{2}(v - u, \nabla^2 j(\xi)(v - u))_{L^2(0,T)}$ holds. Below we show that this result
can be recovered in the absence of the assumption of continuity of
$w \mapsto \nabla^2 j(w)$. This has been
observed in \cite[Section 3.2]{marko2022integer}, where a slightly 
different formulation of the Taylor expansion is used.
\footnote{The authors thank Gerd Wachsmuth 
for the hint to Darboux's theorem and \cite{apostol1974mathematical}.}
\begin{proposition}\label{prp:improved_taylor}
Let $j : L^2(0,T) \to \R$ be twice Fr\'{e}chet differentiable. 
Let $w$, $\bar{w} \in L^2(0,T)$ be given. Then there exists
$\xi = w + \tau (\bar{w} - w)$ for some $\tau \in [0,1]$
such that $j(\bar{w}) = j(w) + (\nabla j(w), \bar{w} - w)_{L^2(0,T)} +
\tfrac{1}{2}(\bar{w} - w, \nabla^2 j(\xi)(\bar{w} - w))_{L^2(0,T)}$.
\end{proposition}
\begin{proof}
We reduce the problem to the finite-dimensional case by considering the function
$\tilde{j} : [0,1] \ni t \mapsto j(w + t (\bar{w} - w)) \in \R$. The chain rule
in Banach spaces \cite[Theorem 4.D]{zeidler2012applied}
implies that $\tilde{j}''(t) = (\bar{w} - w, \nabla^2 j(w + t(\bar{w} - w))(\bar{w} - w))_{L^2(\Omega)}$ for $t \in [0,1]$. Then we apply
a variant of Taylor's theorem that is based on Darboux's theorem
and does not require continuity of the second derivative \cite[Theorem 5.19]{apostol1974mathematical},
to obtain that there exists $\tau \in [0,1]$ such that
$j(\bar{w}) = j(w) + (\nabla j(w), \bar{w} - w)_{L^2(\Omega)} +
\tfrac{1}{2}(\bar{w} - w, \nabla^2 j(w + \tau(\bar{w} - w))(\bar{w} - w))_{L^2(\Omega)}$,
which proves the claim with the choice $\xi \coloneqq w + \tau(\bar{w} - w)$.
\end{proof}
From our point of view, the most restrictive part of \cref{ass:strong_assumption} is
the boundedness of the bilinear form with respect to the product of the $L^1$-norms of the inputs, which
is required for the convergence analysis of \Cref{alg:slip} in \cite{leyffer2022sequential,manns2023integer}.
This is because in the setting of $W$-valued controls, $|W| < \infty$, the authors of \cite{leyffer2022sequential,manns2023integer}
are able to construct functions $d$ such that in a small neighborhood $(\nabla j(w), d)_{L^2}$ decreases at least linearly with respect
to $\|d\|_{L^1}$ and the $w + d$ are feasible if $w$ is not stationary. However, $|W| < \infty$ also implies $\|d\|_{L^1} = \Theta(\|d\|_{L^2}^2)$,
even $\|d\|_{L^1} = \|d\|_{L^2}^2$ if $W = \{0,1\}$. Therefore, in order to dominate the quadratic term in the proofs, the
boundedness with respect to the product of the $L^1$-norms is assumed, see also the related comments in \cite{leyffer2022sequential,manns2023integer}.
A related assumption (Lipschitz continuity of the derivative) is made in (5), (10) in \cite{hahn2023binary} (note that the abstract space $Y$ therein
becomes $L^1$ for the examples).
However, \cref{ass:strong_assumption} may be considered to be too restrictive for practical
applications of \Cref{alg:slip} because it requires an improvement of the input regularity
of the control-to-state operator.
In particular, as we also experience for the considered poroviscoelastic
problem in \Cref{sec:biots_model}, one may only be able to verify one
of the weaker assumptions below, where the uniform boundedness of the 
Hessian of $j$ is assumed with respect to stronger norms for the
control input.
\begin{assumption}\label{ass:weak_assumption}
Let $j : L^2(0,T) \to \R$ be twice Fr\'{e}chet differentiable such that for all $w \in L^2(0,T)$
\[ |\nabla^2 j(w)(\psi,\phi)| \le C \|\psi\|_{L^2(0,T)}\|\phi\|_{L^2(0,T)} \]
holds for some $C>0$ and all $\psi,\phi \in L^2(0,T)$.
\end{assumption}

\begin{assumption}\label{ass:weak_assumption_H1}
Let $j : H^1(0,T) \to \R$ be twice Fr\'{e}chet differentiable
such that for all $w \in H^1(0,T)$
\[ |\nabla^2 j(w)(\psi,\phi)| \le C \|\psi\|_{H^1(0,T)}\|\phi\|_{H^1(0,T)} \]
holds for some $C>0$ and all $\psi,\phi \in H^1(0,T)$.
\end{assumption}

\medskip

\underline{Regularization of control inputs.}  
Let $j_\varepsilon$ be defined as follows: 
\begin{equation}\label{jep}
 j_\varepsilon \coloneqq j \circ K_\varepsilon,
\end{equation}
where $K_\varepsilon :  L^1(0,T) \to L^2(0,T)$ is
a (bounded and linear) convolution operator defined as
\begin{equation}\label{Kep}
K_\varepsilon(w) \coloneqq r_{[0,T]}(\eta_\varepsilon * w), \ \text{for any}\ w \in L^1(0,T),
\end{equation}
where $(\eta_\varepsilon)_{\varepsilon > 0}$ is a family of positive 
mollifiers \cite[Section 4.5]{maggi2012sets}, and $r_{[0,T]}$ denotes the
restriction of a function defined on all of $\R$ to the interval 
$[0,T]$. For the sake of the convolution being well-defined we assume 
that all $w \in L^1(0,T) $ are extended to $0$ outside of $[0,T]$ when 
passed into the convolution operation.

Then the following three propositions hold.
\begin{proposition}\label{prp:relax_assumption}
Let $\varepsilon > 0$. Then $K_\varepsilon : L^1(0,T) \to L^2(0,T)$  defined above in \eqref{Kep}
is a bounded linear operator. If $j$ satisfies \cref{ass:weak_assumption},
then $j_\varepsilon$ satisfies \Cref{ass:strong_assumption} with some $C = C_\varepsilon$ that depends on $\varepsilon$.
\end{proposition}
\begin{proof}
Let $w \in L^1(0,T)$. Let $K_\varepsilon^* : L^2(0,T) \to L^\infty(0,T)$
denote the adjoint operator of the bounded linear operator $K_\varepsilon$,
where we have identified $(L^2(0,T))^* \cong L^2(0,T)$
and $(L^1(0,T))^* \cong L^\infty(0,T)$.
The chain rule, see, \cite[Theorem 4.D]{zeidler2012applied},
yields the derivatives
\[\nabla j_\varepsilon(w) = K_\varepsilon^* \nabla j(K_\varepsilon w)
\quad \text{and}\quad \nabla^2 j_\varepsilon(w) = 
\langle K_\varepsilon^* \nabla^2 j(K_\varepsilon w)K_\varepsilon \cdot, 
\cdot \rangle_{L^\infty(0,T), L^1(0,T)}.\]
Let $\phi$, $\psi \in L^1(0,T).$  Cauchy--Schwarz inequality
and the submultiplicativity of the operator norm give
\[ \langle K_\varepsilon^* \nabla^2 j(K_\varepsilon w)K_\varepsilon \psi, 
\phi \rangle_{L^\infty(0,T), L^1(0,T)}
= (\nabla^2 j(K_\varepsilon w)K_\varepsilon \psi, 
K_\varepsilon \phi)_{L^2(0,T)}
\le C \|K_\varepsilon \psi\|_{L^2(0,T)}\|K_\varepsilon \phi\|_{L^2(0,T)}
\]
where $C > 0$ is the constant from \Cref{ass:weak_assumption}.
By virtue of Young's convolution inequality, we obtain 
$\|K_\varepsilon \phi\|_{L^2(0,T)}
\le \|\eta_\varepsilon\|_{L^2(\R)}\|\phi\|_{L^1(0,T)}$.
Because $(\eta_\varepsilon)_{\varepsilon > 0}$ is a family of mollifiers
we have that $\|\eta_\varepsilon\|_{L^2(\R)} < \infty$ for all
$\varepsilon > 0$, and thus the claims follow.
\end{proof}
\begin{proposition}\label{prp:relax_assumption_H1}
Let $\varepsilon > 0$. Then $K_\varepsilon : L^1(0,T) \to H^1(0,T)$  defined above in \eqref{Kep}
is a bounded linear operator. If $j$ satisfies \Cref{ass:weak_assumption_H1},
then $j_\varepsilon$ satisfies \Cref{ass:strong_assumption} with some $C = C_\varepsilon$ that depends on $\varepsilon$.
\end{proposition}
\begin{proof}
The proof is very similar to that of \Cref{prp:relax_assumption} above. 
We need to show that
\[ \|K_\varepsilon \phi\|_{H^1}
   = \left(\|K_\varepsilon \phi\|_{L^2}^2 + 
           \left\|\frac{\dd}{\dd t} K_\varepsilon \phi\right\|_{L^2}^2\right)^{\frac{1}{2}}
   \le C \|\phi\|_{L^1}, \]
for some $C > 0$. From \Cref{prp:relax_assumption} we  already have that
$\|K_\varepsilon \phi\|_{L^2} \le c_1  \|\phi\|_{L^1}$. Now using the formula for the derivative of a convolution $\frac{\dd}{\dd t} (\eta_\varepsilon * \phi)
= \left(\frac{\dd}{\dd t} \eta_\varepsilon\right) * \phi$ and Young's convolution inequality, we obtain that  
$\left\|\frac{\dd}{\dd t}K_\varepsilon \phi\right\|_{L^2} \le c_2  \|\phi\|_{L^1}$ holds with
$c_2 \coloneqq \left\|\frac{\dd}{\dd t} \eta_\varepsilon\right\|_{L^2}$,
which is bounded because $\eta_\varepsilon$ is smooth.
\end{proof}

\begin{proposition}\label{prp:smoothing_grad}
Let $(\eta_\varepsilon)_{\varepsilon > 0}$ be a 
family of standard mollifiers. Let  $j_\varepsilon$ be defined as above 
in \eqref{jep}. Let $w \in L^2(0,T)$. Then
$\nabla j_\varepsilon(w) = K_\varepsilon^* \nabla j(K_\varepsilon w) 
\in L^\infty(0,T)$ is a continuous function on $[0, T]$.
\end{proposition}
\begin{proof}
As above, the identification
$L^\infty(0,T) \cong (L^1(0,T))^*$ implies that we can
consider $K_\varepsilon^* \nabla j(K_\varepsilon w)$ as an $L^\infty(0,T)$-function. Moreover, $\nabla j(K_\varepsilon w)$
is an $L^2(0,T)$-function. Let $f \in L^1(0,T)$, $g \in L^2(0,T)$.
We consider the adjoint operator $K_\varepsilon^*$, which is defined by
the identity
\[ (K_\varepsilon f, g)_{L^2( 0,T )} = \langle f, K_\varepsilon^* g\rangle_{L^1(0,T),L^\infty(0,T)}.
\] 
We insert the definition of the convolution and obtain the following identity
\[ \int_0^T \overbrace{\int_0^T \eta_\varepsilon(t - s)f(s)\dd s}^{(K_\varepsilon f)(t)}g(t)\dd t = \int_0^T f(s) \overbrace{\int_0^T \eta_\varepsilon(t - s)g(t)\dd t}^{(K_\varepsilon^*g)(s)} \dd s
\]
by virtue of Fubini's theorem. $K_\varepsilon f \in C^\infty([0,T])$ 
holds as $\eta_\varepsilon * f$ is the convolution of $f$ and a 
mollifier, see Theorem 7 in \cite{evans2010partial}.
Due to the structural symmetry of the mollification (or also of the 
mollifiers themselves), the proof for smoothness of $K_\varepsilon f$
on $[0,T]$ can be transferred directly to $K_\varepsilon^* g$ on $[0,T]$.  
\end{proof}
Let $(P_{\varepsilon})$ be the optimization problem \eqref{eq:p} where 
the objective $j$ is replaced by $j_\varepsilon$. We note that the constant $C_\varepsilon$ asserted by 
\Cref{prp:relax_assumption,prp:relax_assumption_H1} blows up for $\varepsilon \searrow 0$
so that the property can not be carried over to the limit.
We are interested in the ability of local solutions, global solutions, and
stationary points of the new optimization problem $(P_{\varepsilon})$
to approximate local solutions, global solutions, and
stationary points of \eqref{eq:p} as $\varepsilon \to 0$.
While we are not able to give a full answer to the question at this 
point, we can provide a positive answer in  the case of global solutions 
by virtue of $\Gamma$-convergence. 
Assume that the objectives ($j_\varepsilon + \alpha \TV$) 
$\Gamma$-converge to ($j + \alpha \TV$) as $\varepsilon \to 0$ with 
respect to $\BV(0,T)$-weak-$^*$-convergence on the complete subspace 
$\BVW(0,T)$, which is the feasible set
of \eqref{eq:p}. Then we obtain that global minimizers of $j_\varepsilon$ in $\BVW(0,T)$
weakly-$^*$-converge to global minimizers of $j$ in $\BVW(0,T)$, which is
one of our main results.
\begin{theorem}\label{thm:gamma_convergence}
Let $j$ be continuous.
For $\varepsilon > 0$, let $(\eta_\varepsilon)_{\varepsilon > 0}$ be a
family of positive mollifiers. Let $K_\varepsilon$ be defined as above in \eqref{Kep}. 

Then the sequence $(\hat{j}_{\varepsilon^n})_{n\in \N}$,
defined as $\hat{j}_{\varepsilon^n} \coloneqq j_{\varepsilon^n} + \alpha \TV$, 
$\Gamma$-converges to $\hat{j} = j + \alpha \TV$ as $\varepsilon^n \to 0$
with respect to weak-$^*$-convergence and strict convergence in $\BVW(0,T)$.
\end{theorem}
\begin{proof}
We start by proving the lower bound inequality. To this end, let
$w^{\varepsilon^n} \weakstarto w$ in $\BVW(0,T)$. Then
$\TV(w) \le \liminf_{\varepsilon^n \to 0} \TV(w^{\varepsilon^n})$
by virtue of the weak-$^*$-lower semicontinuity of the
total variation. 

We now show that $j_{\varepsilon^n}(w^{\varepsilon^n}) \to j(w)$. Due to the continuity of $j$ it suffices to show that 
\[\| K_{\varepsilon^n} w^{\varepsilon^n} - w\|_{L^2(0,T)} \to 0. \]  We observe that
\[ \| K_{\varepsilon^n} w^{\varepsilon^n} -w \|_{L^2(0,T)} \leq  \|K_{\varepsilon^n} w^{\varepsilon^n} - K_{\varepsilon^n} w\|_{L^2(0,T)} + \|K_{\varepsilon^n} w - w\|_{L^2(0,T)}  .\]
Using Young's convolution inequality we obtain 
\[ \|K_{\varepsilon^n} w^{\varepsilon^n} - K_{\varepsilon^n} w\|_{L^2(0,T)} = \|K_{\varepsilon^n} (w^{\varepsilon^n} -  w)\|_{L^2(0,T) } \leq \| \eta_{\varepsilon^n} \|_{L^1(\mathbb{R})} \|w^{\varepsilon^n} -  w\|_{L^2(0,T)} \leq \|w^{\varepsilon^n} -  w\|_{L^2(0,T)} \to 0. \]
The convergence $  \|K_{\varepsilon^n} w - w\|_{L^2(0,T)} \to 0$ follows from Theorem 7, page 714, in \cite{evans2010partial}.

To prove the upper bound inequality, we choose $w^{\varepsilon^n} \coloneqq w$ for all $n \in \N$ and observe
\[ \hat{j}_{\varepsilon^n}(w^{\varepsilon^n})
= j_{\varepsilon^n}(w) + \alpha\TV(w)
\to j(w) + \alpha\TV(w)
\]
with the same argument as above. This proves $\Gamma$-convergence
with respect to weak-$^*$-convergence in $\BVW(0,T)$.
Because the chosen sequence for the upper bound inequality
is also strictly convergent, $\Gamma$-convergence also holds
with respect to strict convergence in $\BVW(0,T)$.
\end{proof}

\Cref{prp:relax_assumption} shows it is sufficient to
verify the much weaker assumptions \Cref{ass:weak_assumption}
or \Cref{ass:weak_assumption_H1}
instead of \Cref{ass:strong_assumption}
when we solve \eqref{eq:p} with the input regularization 
$j_\varepsilon$ instead of $j$ for some $\varepsilon > 0$. 
\Cref{prp:smoothing_grad} implies that $\nabla j_\varepsilon(w)$, with
$w \in \BVW(0,T)$, is a smooth function on the interval $[0,T]$. 
Thus we can use the characterization
\[ \nabla j_\varepsilon(w)(t_i) = 0 \]
for all of the finitely many $t_i$ where the function $w$ has a jump,
if $w$ is L-stationary for \eqref{eq:p} with $j_\varepsilon$ instead of $j$,
see \cite{leyffer2022sequential}.

While we obtain convergence of global minimizers
under \Cref{ass:weak_assumption}, we do not know at 
present if we obtain converge  to an L-stationary point of  
\eqref{eq:p} when we compute
L-stationary points for \eqref{eq:p} with  $j_\varepsilon$ instead of $j$
and drive $\varepsilon$ to zero. In particular,
we have not been able to show that the limit of 
L-stationary or approximately L-stationary points of \eqref{eq:p} with 
$j_\varepsilon$ that are produced by \Cref{alg:slip} for a homotopy that drives
$\varepsilon \to 0$ is L-stationary for \eqref{eq:p}.
We note that even if this were true, L-stationary is
not known to be a necessary optimality condition for 
\eqref{eq:p} if only \Cref{ass:weak_assumption}
but not \Cref{ass:strong_assumption} holds.

Moreover, the situation is even worse in case that we only have 
\Cref{ass:weak_assumption_H1}.
Because $H^1(0,T) \hookrightarrow C([0,T])$,
the only $W$-valued functions in $H^1(0,T)$ are constant on the whole 
the domain $(0,T)$, so we cannot prove a result like 
\Cref{thm:gamma_convergence} in this case and the limit problem 
$\varepsilon = 0$ has no meaningful interpretation due to the
high regularity that is required for the control input.
However, in case one is still interested in discrete-valued
controls even if this is not covered by the regularity theory
for the PDE, we believe that solving \eqref{eq:p} with $j_\varepsilon$
instead of $j$ is still sensible because for a given control
$w \in \BVW(0,T)$ we know that there are only finitely many
switches (or jumps) that can occur and by means of the parameter
$\varepsilon$ we can control the support size of the smooth
transitions between them that occurs when mollifying them.

However, the succeeding analysis shows that in case of 
\Cref{ass:weak_assumption}, weak-$^*$ limit points of the 
homotopy are also strict limit points, which means that the limit cannot
have a lower total variation than its approximating sequence.
Consequently, there is no nearby reduction of the objective function by
an improvement of the $\TV$-term.
This result is possible even in the presence of realistic early termination
criteria in \Cref{alg:slip}.

\begin{proposition}\label{prp:finite_termination_slip_w_extra_termination}
Let $j$ be bounded below. Assume that \Cref{alg:slip} is terminated
when one of the following conditions is met: 
\begin{itemize}
    \item the trust-region radius is smaller than a 
    given $C_1>0$
    \item the predicted reduction $-\ell(w^{n-1},d^{n,k})$ in outer iteration
    $n \in \N$ and inner iteration $k \in \N$
    is smaller than a given $C_2>0$,
\end{itemize}
then it terminates within finitely many outer
iterations regardless of the initial guess $w^0$.
\end{proposition}

\begin{remark}\label{rem:termination_criteria}
Before proving \Cref{prp:finite_termination_slip_w_extra_termination},
we make a brief note to explicitly explain how these conditions enter 
\Cref{alg:slip}. The first condition is checked after each reduction of the trust-region radius together with \Cref{alg:slip}
ln.\ \ref{ln:inner_loop_condition}.
The second condition replaces the termination criterion
in \Cref{alg:slip} ln.\ \ref{ln:terminate}.
\end{remark}

\begin{proof}[Proof of \Cref{prp:finite_termination_slip_w_extra_termination}]
The first condition ensures that the inner loop iterates
only finitely many times, specifically
at most $\lfloor \log_2(C_1 / \Delta^0)\rfloor$ times.
Assume that \Cref{alg:slip} does not terminate within
finitely many iterations of the outer loop. Then the inner
loop accepts a new iterate within
$k \le \lfloor \log_2(C_1 / \Delta^0)\rfloor$ iterations
in each outer iteration $n \in \N$.
The step acceptance implies that the reduction in the 
objective is always higher than $-\sigma \ell(w^{n-1},d^{n,k})$.
It follows from the second termination criterion that
$-\ell(w^{n-1},d^{n,k}) \ge C_2$
for all $n \in \N$ and corresponding $k$ on acceptance.
This contradicts
that $j$ and $\TV$ are bounded from below. Consequently,
\Cref{alg:slip} terminates within finitely many
outer iterations.
\end{proof}

\begin{proposition}\label{prp:final_homotopy}
Let $j$ be bounded below.
Let \Cref{ass:weak_assumption} hold.
Let $(w_\varepsilon)_\varepsilon \subset \BVW(0,T)$
be the sequence of final iterates produced by \Cref{alg:slip}
executed on $j_\varepsilon$
for a sequence $\varepsilon \to 0$ with the initial control given by the previous  final iterate,
where the execution of \Cref{alg:slip} is terminated
if one of following conditions is met:
\begin{itemize}
    \item the trust region radius is smaller than $\Delta(\varepsilon)>0$, which tends to zero as $\varepsilon$ is driven to zero,
    \item the predicted reduction is smaller than $(1-\frac{3a}{4})\alpha>0$ for
    some fixed $0<a\leq 1-\sigma <1$.
\end{itemize}
If $\liminf_{\varepsilon \to 0} \TV(w_\varepsilon) < \infty$,
there is at least one weakly-$^*$ convergent subsequence.
Every weakly-$^*$ convergent subsequence converges strictly
to a limit point in $\BVW(0,T)$. 
\end{proposition}
\begin{proof}
First, we note that the sequence $(w_\varepsilon)_\varepsilon$
is well defined because \Cref{prp:finite_termination_slip_w_extra_termination}
asserts that \Cref{alg:slip} terminates after finitely many iterations
for the two assumed termination criteria.
Second, the facts that $\liminf_{\varepsilon \to 0}\TV(w_\varepsilon) < \infty$
and that $(w_\varepsilon)_\varepsilon$ is bounded in $L^\infty(0,T)$
and hence $L^1(0,T)$ imply that there is at least one weakly-$^*$ convergent subsequence.

In the remainder of the proof we consider an arbitrary weak-$^*$ convergent 
subsequence $w_i \weakstarto \bar{w}$ for $\varepsilon_i \to 0$
in $\BV(0,T)$.
The convergence theory of \Cref{alg:slip} in \cite{leyffer2022sequential}
gives $(w_\varepsilon)_\varepsilon \subset \BVW(0,T)$ and
the weak-$^*$ closedness of $\BVW(0,T)$
(see \Cref{sec:primer}) gives $\bar{w} \in \BVW(0,T)$.
We note that
the subsequence also converges in $L^2(0,T)$ by
boundedness in $L^\infty(0,T)$ (due to $w_i(t) \in W$ a.e.\ for all $\varepsilon_i$) and pointwise convergence a.e.\ of a subsequence.

We set forth to prove the claim by way of contradiction. To this end,
assume that the convergence of the subsequence is not strict.
Then we can find a subsequence, for ease of notation also denoted by
$w_i$, and $\bar C>0$ such that
\begin{align}
~\alpha \TV(w_i)-\alpha \TV(\bar w) \geq \bar C\geq \alpha >0 
\label{eq:barC}
\end{align}
for all $i\in \mathbb{N}$ because the $\TV$-seminorm is weak-$^*$ lower 
semi-continuous. Because $\TV(w_i), \TV(\bar w) \in \mathbb{Z}$,
we can assume that $\bar C \geq \alpha$ holds. Note
\begin{align*}
    |j(K_{\varepsilon_i}w_i)-j(\bar w)| \leq |j(K_{\varepsilon_i} w_i) - j(K_{\varepsilon_i} \bar w)|
    +|j(K_{\varepsilon_i} \bar w)-j(\bar w)|. 
\end{align*}

Additionally, $\|K_{\varepsilon_i} \bar w - \bar w\|_{L^2} \to 0$  
holds by the argumentation provided in Theorem 7, page 714, in 
\cite{evans2010partial}.  Young's convolution inequality shows
\begin{align*}
    \|K_{\varepsilon_i} w_{i} - K_{\varepsilon_i} \bar w\|_{L^2}
    \le \|\eta_{\varepsilon_i}\|_{L^1}
        \|w_{i} - \bar w\|_{L^2} = \|w_i - \bar w\|_{L^2} \to 0.
\end{align*}
Thus, we obtain that
$j(K_{\varepsilon_i} w_i) \to j(\bar w)$ for $i \to \infty.$

 Let $\tilde{\Delta} > 0$ be arbitrary but fixed.
Then there is $i_0 \in \N$ large enough such that for all $i \ge i_0$
\begin{align*}
    \|w_{i} - \bar w\|_{L^1}
    \leq  \tilde \Delta\quad \text{and}\quad
    \|\nabla j(w_{i}) -  \nabla j(\bar w)\|_{L^2} \leq \tilde \Delta
\end{align*}
hold. Thus $\bar w - w_i$ is feasible for $\TR(w_{i},\tilde \Delta)$
for all $i \geq i_0$. Let $w_0 \in \BVW(0,T)$ be an arbitrary point 
	such that $w_0 - w_i$ is feasible for $\TR(w_{i},\tilde \Delta)$.
Using Taylor's theorem we obtain that
\begin{align*}
    j(K_{\varepsilon_i} w_{i})-j(K_{\varepsilon_i}w_0)
    &= \sigma (\nabla j(K_{\varepsilon_i} w_{i}), K_{\varepsilon_i} (w_{i}-w_0))_{L^2}
    + (1- \sigma) (\nabla j(K_{\varepsilon_i} w_{i}), K_{\varepsilon_i} (w_{i}-w_0))_{L^2}  \\
     \MoveEqLeft[-1] +\frac{1}{2} \nabla^2 j(\xi_i) (K_{\varepsilon_i} (w_{i}-w_0),K_{\varepsilon_i} (w_{i}-w_0))
\end{align*}
for some $\xi_i \in L^2(0,T)$. From \Cref{ass:weak_assumption} we derive that
\begin{align*}
    |\nabla^2j(\xi_i) (K_{\varepsilon_i}(w_{i}-w_0),K_{\varepsilon_i}(w_{i}-w_0))| \leq C \|w_{i}-w_0\|_{L^2}^2 \leq C |\max(W)-\min(W)| \|w_{i}-w_0\|_{L^1} \leq C_2 \tilde \Delta
\end{align*}
for some $C$, $C_2>0$.
Furthermore, we obtain that 
\begin{align*}
    |(\nabla j(K_{\varepsilon_i} w_{i}), K_{\varepsilon_i} (w_{i}-w_0))_{L^2} | &\leq \| \nabla j(K_{\varepsilon_i} w_{i}) \|_{L^2} \|K_{\varepsilon_i} (w_{i}-w_0)\|_{L^2} \\
    &\leq (\|\nabla j(\bar w)\|_{L^2}+C_3\tilde \Delta) \|w_{i}-w_0\|_{L^2} \\
    & \leq C_4 (C_5 + C_3 \tilde \Delta) \sqrt{\tilde \Delta}
\end{align*}
holds for some $C_3$, $C_4$, $C_5 > 0$.
Then there exists $\tilde{\Delta}$ as above such that
for $i_0$ corresponding to $\tilde{\Delta}$ we obtain for all $i \ge i_0$
that the estimates
\begin{align}
    \Delta(\varepsilon_i) &<\frac{\tilde \Delta}{2}, \\
    |\nabla^2j(\xi_i) (K_{\varepsilon_i}(w_{i}-w_0),K_{\varepsilon_i}(w_{i}-w_0))| 
    &\leq \frac{a}{4} \bar C \label{Eq: bound second deriv 2} \text{ and }\\
     |(\nabla j(K_{\varepsilon_i} w_{i}), K_{\varepsilon_i} (w_{i}-w_0))_{L^2} | 
     &\le \frac{a}{4} \bar C \label{Eq: bound first deriv 2}
\end{align}
hold for all $w_0$ in the trust region of $\TR(w_{i},\Delta)$ if $\Delta \le \tilde{\Delta}$.
Moreover, there exists a minimal $\tilde{k} \in \N$ such that
$\tilde{\Delta} \ge \Delta^02^{-\tilde{k}}$.

We now show that the inner loop of \Cref{alg:slip} accepts the step not later than
in inner iteration $\tilde{k}$ for $i$ large enough. Specifically,
for all $i \ge i_0$, the execution of \Cref{alg:slip} for $\varepsilon_i$ 
accepts a step not later than in inner iteration $\tilde{k}$ in the final outer iteration.
The trust-region radius upon acceptance is  $\Delta^0 2^{-\tilde{k}}$ with
$\Delta(\varepsilon_i)<\Delta^0 2^{-\tilde{k}} \le \tilde{\Delta} \le \Delta^0 2^{-\tilde{k}+1} $.

Let $w^* - w_i$ be the optimal solution of $\TR(w_{i},\Delta^0 2^{-\tilde{k}})$.
Because $\bar w - w_i$ is in the trust region of
$\TR(w_{i}, \Delta^0 2^{-\tilde{k}})$ it follows that
\begin{align}
    \alpha (\TV(w_{i}) - \TV(w^*)) \geq  \alpha (\TV(w_{i}) - \TV(\bar w)) - \frac{a}{2} \bar C \geq \left(1-\frac{a}{2}\right) \bar C \label{Eq: lower bound TV 2}
\end{align}
because the predicted reduction of 
$w^*$ is greater than or equal to that of $\bar w$.
Let $M \coloneqq \alpha (\TV(w_{i}) - \TV(w^*)) - (1-\frac{a}{2}) \bar C \geq 0$.
In total we obtain the inequalities
\begin{align*}
     \hspace{2em}
     &\hspace{-2em}
     j(K_{\varepsilon_i} w_{i})-j(K_{\varepsilon_i}  w^*) + \alpha (\TV(w_{i}) - \TV(w^*))\\
     &= \sigma (\nabla j(K_{\varepsilon_i} w_{i}), K_{\varepsilon_i} (w_{i}-w^*))_{L^2}
     + (1- \sigma) (\nabla j(K_{\varepsilon_i} w_{i}), K_{\varepsilon_i}(w_{i}-w^*))_{L^2} \\ 
     &\hphantom{=}\ 
     + \frac{1}{2} \nabla^2 j(\xi) (K_{\varepsilon_i} (w_{i}-w^*),K_{\varepsilon_i} (w_{i}-w^*))  + \alpha (\TV(w_{i}) - \TV(w^*))
     \\
     &\geq  -\sigma \ell(w_{i},w^*-w_{i}) - \frac{a}{4} \bar C - \frac{a}{4} \bar C + (1-\sigma) 
     \left(1-\frac{a}{2}\right) \bar C + (1-\sigma) M
          &&\text{\scriptsize \eqref{Eq: bound second deriv 2},\eqref{Eq: bound first deriv 2}}\\
     &\geq  -\sigma \ell(w_{i},w^*-w_{i}) - \frac{a}{4} \bar C - \frac{a}{4} \bar C + \frac{a}{2} \bar C + (1-\sigma) M &&\text{\scriptsize $a \le 1 - \sigma$} \\
     &\geq -\sigma  \ell(w_{i},w^* - w_{i}) \\
     & \geq  \sigma \left(1-\frac{3a}{4}\right) \bar C
     \underset{a < 1}> 0.
     &&\text{\scriptsize \eqref{Eq: bound first deriv 2},\eqref{Eq: lower bound TV 2}}
\end{align*}
We analyze these inequalities with respect to the two possible termination criteria that are assumed (note that the original termination criterion
in \Cref{alg:slip} ln.\ \ref{ln:terminate} is replaced by
the second one, see also \Cref{rem:termination_criteria}).

The first termination criterion does not apply, because the trust-region 
algorithm would find an improvement of $\sigma (1-\frac{3a}{4})\bar C$
before the critical trust-region radius is attained.
The second termination criterion does not apply,
because the predicted reduction would have to be less than 
$(1-\frac{3a}{4})\alpha \le (1-\frac{3a}{4})\bar C$ due to inequality \ref{eq:barC},
which can not happen as the inequalities show.

Thus neither of the assumed termination criteria is satisfied for any
$i \ge i_0$ but the $w_i$ are final iterates of \Cref{alg:slip}
under the assumed termination criteria, which is
a contradiction.
Consequently, the convergence of $w_i$ to $\bar{w}$ is strict.
\end{proof}

\begin{remark}
We note that the proof does not require the assumption
that an execution of \Cref{alg:slip} is initialized
with the final iterate of the execution for the previous
choice of $\varepsilon$. We have included it here
because it is the natural choice for a homotopy that
drives $\varepsilon$ to zero and generally helps to obtain
a bounded (and improving) sequence in practice as can also
be seen in our numerical examples.

We note that a there are ample ways to explore more sophisticated
and more efficient homotopy methods with adaptive choices of $\varepsilon^n$ so that
the homotopy becomes integrated into Algorithm 1.
	
We believe that if the trust-region subproblems \eqref{eq:tr} are solved
inexactly, the convergence analysis in \Cref{prp:final_homotopy} as well as in
\cite{leyffer2022sequential} can still be carried out if one can guarantee that
the inexact solution to \eqref{eq:tr} has an objective value that is smaller than
optimal objective value multiplied by a fixed constant $c \in (0,1)$. Clearly,
$c$ will appear in the constants in the arguments and statements in this case.
Moreover, we believe this can be further combined with inexact evaluations of
$\nabla j_\varepsilon$ when this inexactness is driven to zero over the course
of the iterations.
\end{remark}

\section{Application to Fluid Flows through Deformable,  Porous Media}\label{sec:biots_model}

\textbf{PDE Model.} Let $\Omega \subset \mathbb{R}^3$ be the open, bounded domain occupied by the fluid-solid mixture, with Lipschitz boundary $\partial \Omega$. Motivated by applications in biomechanics (like tissue perfusion \cite{Araujo,frijns,Klisch,Lemon2006,Tosin, causin}), we work under the assumptions of full saturation, negligible inertia, small deformations and incompressible mixture components (in the sense that the  solid and fluid phases can’t undergo volume changes at the microscale).  Due to the complex composition of biological tissue, which exhibit both elastic and viscoelastic behaviors, we consider both poroelastic and poroviscoelastic systems, where the effective stress tensor is of Kelvin-Voigt type. The extent to which structural viscoelasticity is present in the equations is represented by the parameter $\delta \geq 0$. Therefore, the total stress of the fluid-solid mixture is given by 
$$\mathbf T(\mb{u},p) =  
%\delta (2\mu_v \epsilon(\mb{u}_t) + \lambda_v (\nabla \cdot \mb{u}_t)\mb{I})+2\mu_e \epsilon(\mathbf{u}) + \lambda_e (\nabla \cdot \mathbf{u})\, \mathbf{I} - p \mathbf{I}
\delta \mu_v (\nabla \mb{u}_t + \nabla \mb{u}_t^T) + \delta \lambda_v (\nabla \cdot \mb{u}_t)\mb{I}+\mu_e (\nabla \mathbf u + \nabla \mathbf u^T) + \lambda_e (\nabla \cdot \mathbf{u})\, \mathbf{I} - p \mathbf{I},$$ 
where $\mb{u}$ is the elastic displacement  and  $p$ is the fluid pressure. Moreover,  $\mathbf{I}$ stands for  the identity tensor, $\lambda_e$ and $\mu_e$ are the Lam\'e parameters, and $\lambda_v$, $\mu_v$ are the visco-elastic parameters, which are all strictly positive. 

The quasi-static system  is described by two conservation laws: the balance of linear momentum for the fluid-solid mixture and the balance of mass  for the fluid component.
\begin{equation}\label{Eq:balance}
\nabla \cdot \mathbf T(\mb{u},p) + \mathbf F(\mathbf x, t) = \mathbf 0 \quad \mbox{and}\quad  \zeta_{t} + \nabla\cdot\mathbf{v}=S(\mathbf x, t)  \quad \mbox{in}\; \Omega\times (0,T),
\end{equation}
where the fluid content $\zeta$ is given by $\ds  \zeta =  \nabla \cdot \mathbf{u}$, and the discharge velocity $\mathbf{v}$ is given by $\mathbf{v} = - \mathbf{K}\nabla  p,$ with
$\mathbf{K} = k \mathbf{I}$, 
where $\mathbf{K}$ is the permeability tensor and $k$ is a constant.We note here that the formula for the fluid content is a simplification of the more general expression  $ \zeta =  c_0 p + \alpha \nabla \cdot \mathbf{u}$ \cite{biot}, where  $c_0$ is the constrained specific storage coefficient and $\alpha$ is the Biot-Willis coefficient. The simplification is made due to the assumption of incompressible fluid and solid components of the mixture (biological tissues have a mass density close to that of the water), which mathematically translates to $c_0=0$ and $\alpha=1$ \cite{detournay-cheng}, and therefore the fluid content becomes solid dilation.\\

We assume that $\partial \Omega=\Gamma_D \cup \Gamma_N$, where $\Gamma_D$ and $\Gamma_N$ are the Dirichlet and Neumann parts of the boundary (with respect to the elastic displacement), with  $\Gamma_D \cap \Gamma_N = \emptyset$ (while allowing $\overline \Gamma_D \cap \overline\Gamma_N \neq \emptyset$).

We associate the following boundary and initial conditions to the balance laws mentioned above:
\begin{equation}\label{BC1}
\mathbf T(\mb{u},p)\mathbf n=\mathbf g, \quad \mathbf v\cdot \mathbf n =0 \quad \mbox{on}\; \Gamma_N\,,
\end{equation}
\begin{equation}\label{BC2}
\mathbf u = \mathbf 0, \quad p =0 \quad \mbox{on}\; \Gamma_{D,p}\,,
\end{equation}
\begin{equation}\label{BC3}
\mathbf u = \mathbf 0, \quad \mathbf v\cdot \mathbf n =\psi \quad \mbox{on}\; \Gamma_{D,\mb{v}}\,.
\end{equation}
\begin{equation}\label{IC} 
\ds  \mathbf u (x,0) =\mb{u}_0(x)  \ \text{in}\ \Omega.
\end{equation}
Note that the Dirichlet part of the boundary $\Gamma_D = \Gamma_{D,p}\cup\Gamma_{D,\mb{v}}$, where the subscripts $p$ and $\mathbf{v}$ indicate conditions imposed on the Darcy pressure and discharge velocity, respectively. As usual, $\mathbf n$ is the outward unit normal vector.

The data  in the system is represented by the body force per unit of volume $\mathbf{F}$,
the net volumetric fluid production rate $S$, and the boundary sources 
$\mathbf g$ and $\psi$. They can be used as controls. \\

  We impose the following assumptions on the domain:
\begin{assumption}\label{assumpdomain} We assume: 
\begin{enumerate}
\item  $\Gamma_{D}$ is a set of positive measure, so by Korn's inequality:
$$ E(\mb{u}(t)) = \|\nabla \cdot \mb{u}\|^2_{L^2(\Omega)}+ (\nabla \mb{u} : (\nabla \mb{u}+\nabla \mb{u}^T)) \ge c||\mb{u}(t)||^2_{\mb{H}^1(\Omega)}$$
where $\grad \mb{u}$ stands for the Jacobian matrix of $\mathbf u$ and the Frobenius inner product of two matrices is given by
$$ (\mb{A}: \mb{B}) = \sum_{i=1}^3 \sum_{j=1}^3 \int_{\Omega} (A_{ij}B_{ij})d\Omega.$$ 
\item $\Gamma_{D,p}$ is a set of positive measure, so by Poincare's inequality:
$$||v||_{L^2(\Omega)} \le C_P ||\nabla v||_{\mb{L}^2(\Omega)},~~\forall v \in V.$$
\end{enumerate} \end{assumption}

\textbf{Notation.} Let $$\mathbb V \equiv \mathbf V \times V = (H^1_{\Gamma_D}(\Omega))^3 \times  H^1_{\Gamma_{D,p}}(\Omega), $$
where the inner-product on $\mb{V}$ is given by
\begin{equation}\label{bil}a(\mb{u},\mb{w}) = (\nabla \cdot \mb{u}, \nabla \cdot \mb{w}) + (\grad\mb{u}:\grad \mb{w}) +   (\grad \mb{u}: (\grad \mb{w})^T).
\end{equation}
%
%Here  $\grad \mb{u}$ stands for the Jacobian matrix of $\mathbf u$ and the Frobenius inner product of two matrices is given by
%$$ (\mb{A}: \mb{B}) =  \int_{\Omega} (A_{ij}B_{ij})d\Omega.$$ 
We note that the bilinear form $a(\cdot, \cdot)$ defines an inner product on $\mb{V}$, due to \Cref{assumpdomain} on the domain. The inner product for $V$ is inherited from $H^1(\Omega).$\\

We have the following results on well-posedness of weak solutions \cite{BGSW,BS_AWM,BS-AA}: 
\begin{proposition}[Poroelasticity] \label{pe}
Let $\delta=0$, $\mb{F}\in H^1(0,T;\mb{L}^2(\Omega))$, $S\in L^2(0,T;L^2(\Omega))$, and $\mb{g}\in H^1(0,T; \mb{H}^{1/2}(\Gamma_N))$. Additionally, let $\psi(x,t)=w(t)\chi(x)$ where $w(t) \in L^2(0,T)$ and $\chi \in L^2(\Gamma_{D,v})$. Then there exists a unique weak solution
$(\mb{u},p) \in L^2(0,T;\mb{V}) \times L^2(0,T;V)$
to (\ref{Eq:balance}-\ref{IC}). Additionally, the solution satisfies the following energy estimate: 
\begin{equation}\label{EE}\int_0^T\|p\|^2_V+ \|\mb{u}\|^2_\mb{V}dt \leq C(\|\mb{u}_0\|_\mb{V}+\|\mb{F}\|_{H^1(0,T;\mb{L}^2(\Omega))}^2+\|S\|_{L^2(0,T;L^2(\Omega))}^2+\|\mb{g}\|_{H^1(0,T;\mb{H}^{1/2}(\Gamma_N))}^2+\|\psi\|_{L^2(0,T;L^2(\Gamma_{D,v}))}^2)  \end{equation}
\end{proposition}

\begin{proposition}[Poroviscoelasticity] \label{pve}
Let $\delta>0$, $\mb{F}\in L^2(0,T;\mb{L}^2(\Omega))$, $S\in L^2(0,T;L^2(\Omega))$, and\\ $\mb{g}\in L^2(0,T; \mb{H}^{1/2}(\Gamma_N))$. Additionally, let $\psi(x,t)=w(t)\chi(x)$ where $w(t) \in L^2(0,T)$ and $\chi \in L^2(\Gamma_{D,v})$. Then there exists a unique weak solution
$(\mb{u},p) \in H^1(0,T;\mb{V}) \times L^2(0,T;V)$
to (\ref{Eq:balance}-\ref{IC}). Additionally, the solution satisfies the following energy estimate: 
\begin{equation}\label{EEPVE}\begin{split}&\int_0^T\|p\|^2_V+ \|\mb{u}\|^2_\mb{V} +\delta \|\mb{u}_t\|^2_\mb{V} dt\\
&\leq C(\|\mb{u}_0\|_\mb{V}+\|\mb{F}\|_{L^2(0,T;\mb{L}^2(\Omega))}^2+\|S\|_{L^2(0,T;L^2(\Omega))}^2+\|\mb{g}\|_{L^2(0,T;\mb{H}^{1/2}(\Gamma_N))}^2+\|\psi\|_{L^2(0,T;L^2(\Gamma_{D,v}))}^2)  \end{split}\end{equation}
\end{proposition}
Similar results hold when the control $w$ is used as the time portion of $S$. Therefore, we will let $\chi(x)$ be fixed and insert $q=\chi(x)w(t)$ into \eqref{Eq:balance}-\eqref{IC} in place of either $S$ or $\psi$. Let $G : L^2(0,T) \to L^2(0,T;\mb{L}^2(\Omega)) \times L^2(0,T;L^2(\Omega))$ be defined by mapping $w(t)\in L^2(0,T)$ to the unique solution of \eqref{Eq:balance}-\eqref{IC}, embedded in $L^2(0,T;\mb{L}^2(\Omega)) \times L^2(0,T;L^2(\Omega))$, with all sources set to zero except for $q$. Let $(\tilde{\mb{u}},\tilde{p})$ be the unique solution (in the spaces provided in \Cref{pe} and \Cref{pve}) to \eqref{Eq:balance}-\eqref{IC} with the sources set as desired and $q=0$.
We introduce the control problem
\begin{gather}\label{eq:control_problem}
\begin{aligned}
    \min_{w \in L^2(0,T)}\ &
    \frac{1}{2}\|\mb{u} - \mb{u}_d\|_{L^2(0,T;\mb{L}^2(\Omega))}^2
    + \frac{1}{2}\|p - p_d\|_{L^2(0,T;L^2(\Omega))}^2
    + \frac{\lambda}{2}\|w\|_{L^2(0,T)}^2
    + \alpha \TV(w)
    \\
    \text{s.t.}\quad\quad & 
    (\mb{u},p) = 
    Gw+ (\tilde{\mb{u}}, \tilde{p}),\\
    & w(t) \in W \text{ for a.a.\ } t \in (0,T)
\end{aligned}
\end{gather}
for given $\mb{u}_d \in L^2(0,T;\mb{L}^2(\Omega))$,
$p_d \in L^2(0,T;L^2(\Omega))$, $\chi \in L^2(\Gamma_{D,v})$, $\lambda \ge 0$,
and a finite set of integers $W \subset \Z$. We define
\begin{gather}\label{eq:J}
J(\mb{u},p,w) \coloneqq \frac{1}{2}\|\mb{u} - \mb{u}_d\|_{L^2(0,T;\mb{L}^2(\Omega))}^2
    + \frac{1}{2}\|p - p_d\|_{L^2(0,T;L^2(\Omega))}^2
    + \frac{\lambda}{2}\|w\|_{L^2(0,T)}^2
\end{gather}
for $(\mb{u},p,w) \in L^2(0,T;\mb{V}) \times L^2(0,T;V) \times L^2(0,T)$
and the reduced objective 
\begin{equation}\label{reducedobj}
j: L^2(0,T) \to \R, \ j(w) \coloneqq J(Gw+(\tilde{\mb{u}},  \tilde{p}), w)
\end{equation}
%for $w \in L^2(0,T)$. 
\\
Next we verify the applicability of the SLIP algorithm introduced in the previous section, by checking if \Cref{ass:weak_assumption} and \Cref{ass:strong_assumption} are satisfied.
%We obtain the following regularity results.
%
\subsection{Case 1: $\lambda>0$}
In this subsection, we prove that in both poroelastic and poroviscoelastic cases \Cref{ass:weak_assumption} is satisfied while the more stringent \Cref{ass:strong_assumption} is not satisfied.
\begin{proposition}\label{prp:strong_assumption_not_satisfied}
Let $\xi \in L^2(0,T)$. Then it follows that the reduced objective
$j : L^2(0,T) \to \R$ is twice continuously differentiable at $\xi$.
Moreover, $|\nabla^2 j(w)(\xi, \varphi)|=|(G\xi, G\varphi)+\lambda(\xi,\varphi)|$ and there exists $C > 0$ such that
\begin{equation}\label{L2bound} |\nabla^2 j(w)(\xi,\varphi)|
\le C \|\xi\|_{L^2(0,T)}
\|\varphi\|_{L^2(0,T)}
\end{equation}
for all $w \in L^2(0,T)$ and all $\xi$, $\varphi \in L^2(0,T)$ i.e. the Hessian is continuous on $L^2(0,T) \times L^2(0,T)$. 
When $\lambda > 0$, \Cref{ass:strong_assumption} does not hold.%
\end{proposition}
\begin{proof}
Let $\tilde{y}=(\mb{u}_d-\tilde{\mb{u}}, p_d-\tilde{p})$. Based on the definition \eqref{reducedobj} of the reduced functional $j$, we have that
$$\grad j(q)=(G(\cdot), Gq - \tilde{y}) + \lambda (\cdot, q)  \ \text{and}\  
 \nabla^2 j(q)(\xi,\varphi)=(G\varphi, G\xi)+ \lambda(\varphi, \xi)$$
Furthermore, we can estimate the Hessian of $j$ as follows:
\begin{equation}\begin{split} \label{hessian} \left|\nabla^2 j(q)(\xi, \varphi)\right|&=|(G\varphi, G\xi)+ \lambda(\varphi, \xi)|\\
&\leq \|G\varphi\|_{L^2(0,T;\mb{L}^2(\Omega) \times L^2(\Omega))}\|G\xi\|_{L^2(0,T;\mb{L}^2(\Omega) \times L^2(\Omega))}+ \lambda \|\varphi\|_{L^2(0,T)}\|\xi\|_{L^2(0,T)}\\
& \leq (C^2+\lambda)\|\varphi\|_{L^2(0,T)}\|\xi\|_{L^2(0,T)},
\end{split}\end{equation} which provides the desired estimate \eqref{L2bound}.
Additionally, we note that 
\begin{equation} \label{star}\lambda \|\varphi\|^2_{L^2(\Omega)}\leq \lambda (\varphi, \varphi)+(G\varphi, G\varphi)=\nabla^2j(q)(\varphi,\varphi)\quad \text{for all } \varphi \in L^2(0,T).\end{equation}
Consider the sequence $\phi_n(t)=n\left(\frac{t}{T}\right)^n \in L^2(0,T)$. Note that 
\begin{equation}\label{norml1}
\|\phi_n\|^2_{L^1(0,T)}=\left(\int_0^T \left|n\left(\frac{t}{T}\right)^n\right|dt\right)^2=\frac{n^2T^2}{(n+1)^2}
\end{equation}
\begin{equation}\label{norml2}
\lambda\|\phi_n\|^2_{L^2(0,T)}=\lambda \int_0^T \left|n\left(\frac{t}{T}\right)^n\right|^2dt=\frac{\lambda n^2T}{2n+1}.
\end{equation}
Combining \eqref{star} with  \eqref{norml1} and \eqref{norml2} we obtain 
\[\frac{|\nabla^2 j(w) (\phi_n, \phi_n)|}{\|\phi_n\|_{L^1(0,T)}^2} \geq \frac{\lambda \|\phi_n\|^2_{L^2(\Omega)}}{\|\phi_n\|_{L^1(\Omega)}^2}=\frac{\lambda n^2 T}{2n+1} \frac{(n+1)^2}{n^2 T^2}=\lambda \frac{(n+1)^2}{(2n+1)T}\to \infty\quad \text{as }n \to \infty.\]
Hence, \Cref{ass:strong_assumption} does not hold.
\end{proof}
\subsection{Case 2: $\lambda =0$}
Now we consider the case $\lambda = 0$. For the following proofs and numerical results, we consider the problem in one spatial dimension i.e. $\Omega=(0,L).$
Following \cite{MBE18, MBE19}, the one dimensional formulation of the poroviscoelastic systems is given by 
\begin{align}
\label{firsteq} H_v\frac{\partial^3 u}{\partial x^2 \partial t}+H_e\frac{\partial^2 u}{\partial x^2}-\frac{\partial p}{\partial x}&=0 &&\forall(x,t) \in [0,L] \times [0,T]\\
\frac{\partial^2 u}{\partial x \partial t}-k\frac{\partial^2 p}{\partial x^2} &=S &&\forall (x,t) \in [0,L] \times [0,T]\\
u(0,t)=p(0,t)&=0 &&\forall t \in [0,T]\\
-k\frac{\partial p}{\partial x}(L,t)&=\psi(t) &&\forall t \in [0,T]\\
\label{sigmabc}H_v\frac{\partial^2 u}{\partial x\partial t}(L,t)+H_e\frac{\partial u}{\partial x}(L,t)-p(L,t)&=0 &&\forall t \in [0,T]\\
\frac{\partial u}{\partial x}(x,0) &=0 &&\forall x \in [0,L] \label{1DIC}
\end{align}
where $H_e=\lambda_e + 2 \mu_e$ and $H_v=\delta(\lambda_v +2 \mu_v)$.
When $F=0$, we know $p(x,t)=H_e\frac{\partial u}{\partial x} +H_v\frac{\partial^2u}{\partial x \partial t}+h(t)$. 
When $g=0$, applying (\ref{sigmabc}) shows $h(t)$ must be equal to 0.
Therefore, when $F=0$ and $g=0$ we have
\begin{equation} \label{pequal} p(x,t)=H_e\frac{\partial u}{\partial x}+H_v\frac{\partial^2 u}{\partial x \partial t}. \end{equation}
When $\lambda=0$, we will see that \Cref{ass:strong_assumption} holds in the poroelastic case, but not the poroviscoelastic case. We will study these two cases separately.
\subsubsection{Poroelastic Case}
We will first consider the case when $\psi$ is used as the control. Then we will study the case where the control $w(t)$ enters the system in the source $S$, i.e. $S(x,t)=\chi(x)w(t)$.
\begin{theorem}\label{5.4}
Let $\lambda = 0$ and $\delta=0$. Let $k(x,t)=k$ where $k$ is a positive constant and let $G$ map a control $\psi$ to the state $(u,p)$ that satisfies \eqref{firsteq}-\eqref{1DIC} with all sources set to zero except $\psi(t)$. Then \Cref{ass:strong_assumption} holds with $w=\psi$.
\end{theorem}
\begin{proof}
Recall from \eqref{hessian}, that $\nabla ^2 j(\psi)(\xi, \varphi)=(G\xi, G\varphi)$. We will first calculate $G\xi$ when $\xi \in C_0^1(0,T)$. Then we will use the Bounded Linear Extension Theorem to show \Cref{ass:strong_assumption} holds. Note that $\xi$ is only a function of $t$ because we are considering the one-dimensional case where $\Gamma_N$ is only one point.
Since $\delta=0$, $H_v=0$. Plugging \eqref{pequal} into the PDE, we see that $p(x,t)$ needs to satisfy
\begin{align*}
    \frac{\partial p}{\partial t}-kH_e\frac{\partial ^2 p}{\partial x^2}&=0 \quad && \forall (x,t) \in [0,L] \times [0,T]\\
    p(0,t)=0,\quad -k\frac{\partial p}{\partial x}(L,t)&=\xi(t) \quad &&\forall t \in [0,T]\\
    p(x,0)&=0  \quad &&\forall x \in [0,L].
\end{align*}
Let $\rho=p +\frac{x}{k} \xi(t)$. Then, we see
\begin{align*}
    \frac{\partial \rho}{\partial t}-kH_e \frac{\partial ^2 \rho}{\partial x^2}&=\frac{x}{k}\xi'(t)\quad &&\forall (x,t) \in [0,L] \times [0,T]\\
    \rho(0,t)=0, \quad -k\frac{\partial \rho}{\partial x}(L,t)&=0 \quad &&\forall t \in [0,T]\\
    \rho(x,0)&=0\quad &&\forall x \in [0,L].
\end{align*}
Let $\lambda_n=\frac{(2n-1)\pi}{2L}$. Note that $(\sqrt{2/L}\sin(\lambda_n x))_{n \in \mathbb{N}}$ is a complete orthonormal basis. Let 
\[\rho(x,t)=\sum_{n=1}^\infty f_n(t) \sqrt{2/L}\sin(\lambda_n x).\] Then we see $\rho(x,t)$ satisfies all boundary conditions, and we have
\[\sum_{n=1}^\infty f'_n(t)\sqrt{2/L}\sin(\lambda_n x) + \lambda_n^2 kH_e \sum_{n=1}^\infty f_n(t) \sqrt{2/L} \sin(\lambda_n x) =\frac{x}{k} \xi'(t)\]
and 
\[\sum_{n=1}^\infty f_n(0)\sqrt{2/L} \sin(\lambda_n x)=0.\]
Multiplying both sides of these equations by $\sqrt{2/L}\sin(\lambda_n x)$, integrating these equations from $0$ to $L$, and setting $c_n=(\frac{x}{k}, \sqrt{2/L}\sin(\lambda_n x))_{L^2(0,L)}=\frac{(-1)^{n}\sqrt{2L}}{k\lambda_n^2}$, we have
\[f_n'(t)+kH_e\lambda_n^2 f_n(t)=c_n \xi'(t)\quad \text{and} \quad f_n(0)=0.\]
Therefore, $f_n(t)=e^{-kH_e\lambda_n^2t}\int_0^t  c_n e^{kH_e\lambda_n^2\xi} \xi'(\xi) d \xi $.
Hence, we have 
\begin{align*}\rho(x,t)&=\sum_{n=1}^\infty e^{-kH_e\lambda_n^2t}\int_0^t  c_n e^{kH_e\lambda_n^2\xi} \xi'(\xi) d \xi \sqrt{2/L} \sin(\lambda_n x),\\
p(x,t)&=\sum_{n=1}^\infty e^{-kH_e\lambda_n^2t}\int_0^t  c_n e^{kH_e\lambda_n^2\xi} \xi'(\xi) d \xi \sqrt{2/L} \sin(\lambda_n x)-\frac{x}{k}\xi(t).
\end{align*}
Using \eqref{pequal} and recalling that $H_v=0$, we anti-differentiate $p$ with respect to $x$ and enforce the boundary condition $u(x,0)=0$, to see
\[u(x,t)= \sum_{n=1}^\infty \frac{e^{-kH_e\lambda_n^2t}}{H_e\lambda_n}\int_0^t  c_n e^{kH_e\lambda_n^2\xi} \xi'(\xi) d \xi \sqrt{2/L} (1-\cos(\lambda_n x))-\frac{x^2}{2kH_e}\xi(t).\]

 %\begin{align*}
  %   p(x,t)&=\sum_{n=1}^\infty \frac{e^{-kH\lambda_n^2t}}{H\lambda_n} \sqrt{2/L}\sin(\lambda_n x)\int_0^t c_n e^{kH\lambda_n^2 \xi} m^2 \cos(m\xi)d\xi -\frac{mx}{k}\sin(mt)\\
   %  &=\sum_{n=1}^\infty \frac{e^{-kH\lambda_n^2t}}{H\lambda_n}\sqrt{2/L}\sin(\lambda_n x)c_n m^2\frac{e^{kH\lambda_n^2 t}(kH\lambda_n^2 \cos(mt)+m\sin(mt))-kH\lambda_n^2}{k^2H^2\lambda_n^4 + m^2}-\frac{mx}{k}\sin(mt)\\
    % &=\sum_{n=1}^\infty \frac{\sqrt{2/L}\sin(\lambda_n x)}{H\lambda_n}c_n m^2 \frac{kH\lambda_n^2\cos(mt) + m\sin(mt) -e^{-kH\lambda_n^2t}kH\lambda_n^2}{k^2H^2 \lambda_n^4+m^2}-\frac{mx}{k}\sin(mt)
 %\end{align*}

 Using the fact that $c_n=(\frac{x}{k}, \sqrt{2/L}\sin(\lambda_nx))_{L^2(\Omega)}$, $(\sqrt{2/L}\sin(\lambda_nx))_{n \in \mathbb{N}}$ is an orthonormal sequence, Lebesgue Dominated Convergence Theorem, and Parseval's equality, we have
\begin{align*}
   & \|p\|^2_{L^2(0,T;L^2(0,L))}\\
   &=\int_0^T \sum_{n=1}^\infty c_n^2 \left[\left(e^{-kH_e\lambda_n^2t}\int_0^t e^{kH_e\lambda_n^2 \xi}\xi'(\xi) d \xi \right)^2-2e^{-kH_e\lambda_nt}\int_0^t e^{kH_e\lambda_n^2 \xi}\xi'(\xi) d \xi \xi(t) \right]+\frac{x^2}{k^2}\xi^2(t)dt\\
   &=\int_0^T \sum_{n=1}^\infty c_n^2 \left[\left(e^{-kH_e\lambda_n^2t}\int_0^t e^{kH_e\lambda_n^2 \xi}\xi'(\xi) d \xi \right)^2-2e^{-kH_e\lambda_nt}\int_0^t e^{kH_e\lambda_n^2 \xi}\xi'(\xi) d \xi \xi(t)+\xi^2(t) \right]dt.
\end{align*}
Integrating by parts and using $\xi(0)=0$, we see
\begin{align*}
    &\|p\|^2_{L^2(0,T;L^2(0,L))}=\int_0^T \sum_{n=1}^\infty c_n^2 \left[\left(e^{-kH_e\lambda_n^2 t}\left(\xi(t)e^{kH_e\lambda_n^2t}-\int_0^t \xi(\xi) kH_e\lambda_n^2 e^{kH_e\lambda_n^2\xi}d\xi\right)\right)^2 \right.\\
     \MoveEqLeft[-1]\left.-2e^{-kH_e\lambda_n^2 t}\left(\xi(t)e^{kH_e\lambda_n^2t}-\int_0^t \xi(\xi) kH_e\lambda_n^2 e^{kH_e\lambda_n^2\xi}d\xi\right)\xi(t)+\xi^2(t) \right]dt\\
    &=\int_0^T \sum_{n=1}^\infty c_n^2 \left[\xi^2(t) -2 \xi(t)\int_0^t \xi(\xi) k H_e\lambda_n^2 e^{kH_e\lambda_n^2 (\xi-t)} d\xi + \left(\int_0^t \xi(\xi)kH_e\lambda_n^2 e^{kH_e\lambda_n^2(\xi-t)} d \xi\right)^2 -2 \xi^2(t) \right.\\
    \MoveEqLeft[-1]\left.+2\xi(t) \int_0^t \xi(\xi) kH_e\lambda_n^2 e^{kH_e\lambda_n^2 (\xi-t)}d\xi +\xi^2(t)\right]dt=\int_0^T \sum_{n=1}^\infty c_n^2 \left(\int_0^t \xi(\xi) kH_e\lambda_n^2 e^{kH_e\lambda_n^2(\xi-t)}d \xi \right)^2dt\\
    %&= \int_0^T \sum_{n=1}^\infty c_n^2 k^2H^2\lambda_n^4 \left(\int_0^t \xi(\xi) e^{kH\lambda_n^2(\xi-t)}d\xi\right)^2dt\\
   & \leq \int_0^T \sum_{n=1}^\infty c_n^2 k^2 H_e^2 \lambda_n^4 \left(\int_0^T \xi(\xi) e^{-kH_e\lambda_n^2(t-\xi)}d\xi\right)^2 dt \leq \int_0^T \sum_{n=1}^\infty H_e^2 \left(\int_0^T \xi(\xi) e^{-kH_e\lambda_n^2(t-\xi)}d\xi\right)^2 dt.
    \end{align*}
%Recall, Young's inequality for convolution: Given $f \in L^p(0,T)$ and $g\in L^q(0,T)$ such that $\frac{1}{p} + \frac{1}{q}=1+ \frac{1}{r}$ with $1\leq p,q, r \leq \infty$,
%   $ \|f*g\|_{L^r(0,T)} \leq C\|f\|_{L^p(0,T)}\|g\|_{L^q(0,T)}$. 
   Using $p=1$, $q=2$, and $r=2$ in \eqref{Yineq}, we have
   \begin{align*}
\|p\|^2_{L^2(0,T;L^2(\Omega))}&
    \leq \sum_{n=1}^\infty H_e^2 \|e^{-kH_e\lambda_n^2 (\cdot)}\|_{L^2(0,T)}^2\|\xi\|_{L^1(0,T)}^2
    \leq \sum_{n=1}^\infty \frac{ H_e^2}{2kH_e\lambda_n^2}\|\xi(\xi)\|^2_{L^1(0,T)}\\
    &=\frac{H_e}{2k}\|\xi\|_{L^1(0,T)}^2\sum_{n=1}^\infty\frac{4L^2}{(2n-1)^2\pi^2}=\frac{H_eL^2}{4k}\|\xi\|_{L^1(0,T)}^2
\end{align*}
Additionally, integrating \eqref{pequal} with respect to $x$, we have

\begin{align*}
    \|u\|_{L^2(0,T;L^2(0,L))}^2 &=\int_0^T \int_0^L \left(\int_0^x p(\xi,t) d\xi\right)^2 dx dt
    %\leq \int_0^T \int_0^L \|p\|_{L^1(0,L)}^2 dx dt\leq L\|p\|_{L^2(0,T;L^2(0,L)}^2 
    \leq \frac{H_eL^3}{4k}\|\xi\|_{L^1(0,T)}^2
\end{align*}
Therefore, for all $\xi\in C_0^1(0,T)$, we have
\begin{equation} \label{Gxibound}\|G\xi\|^2_{L^2(0,T;L^2(0,L) \times L^2(0,L))}\leq C\|\xi\|_{L^1(0,T)}^2.\end{equation} Hence, $G|_{C_0^1(0,T)}$ is bounded and linear. Using the Bounded Linear Extension theorem, we have that the extension $G:L^1(0,T) \to L^2(0,T;L^2(0,L) \times L^2(0,L))$ is a bounded linear functional. Hence, \eqref{Gxibound} holds for all $\xi \in L^1(0,T)$. Thus, $|\nabla^2 j(\psi)(\xi, \varphi)|=(G\xi, G\varphi)\leq \|G\xi\|_{L^2(0,T;L^2(0,L)\times L^2(0,L))}\|G\varphi\|_{L^2(0,T;L^2(0,L) \times L^2(0,L))} \leq C^2 \|\xi\|_{L^1(0,T)}\|\varphi\|_{L^1(0,T)}$.
\end{proof}
Now we consider the poroelastic case with $\lambda=0$ where the control 
is used in the source $S$. Let $\chi(x)$ be set and consider the case 
where $G$ maps $s\in L^2(0,T)$ to the unique solution $(u,p)$ to 
\eqref{Eq:balance}-\eqref{IC} with all sources and the initial condition 
set to zero except $S \coloneqq \chi(x)s(t)$. In this case, the process 
of finding  $p(x,t)$ is similar to the process of finding $\rho(x,t)$ in 
the proof of \Cref{5.4}. Hence, $Gs=(u,p)$ where
\[p(x,t)=\sum_{n=1}^\infty e^{-kH_e \lambda_n^2 t}\int_0^t d_n e^{kH_e \lambda_n^2 \xi} s(\xi) d \xi\sqrt{2/L}\sin(\lambda_n x),\]
and applying \eqref{pequal}
\[u(x,t)=\sum_{n=1}^\infty \frac{e^{-kH_e \lambda_n^2 t}}{H_e \lambda_n}\int_0^t d_n e^{kH_e \lambda_n^2 \xi} s(\xi)d \xi \sqrt{2/L} \sin(\lambda_n x)\]
where $d_n=(\chi(x), \sqrt{2/L}\sin(\lambda_n x))_{L^2(0,L)}.$ The proof for showing this solution also satisfies \Cref{ass:strong_assumption} follows similarly to the proof of \Cref{5.4}.
\subsubsection{Poroviscoelastic Case}
We show that \Cref{ass:strong_assumption} is not satisfied in the poroviscoelastic case, for both choices of controls $\psi$ and $S$.
\begin{theorem}
Let $\lambda=0$ and $\delta>0$. Let $k(x,t)=k$ where $k$ is a positive constant and let $G$ map a control $\psi$ to the state $(u,p)$ that satisfies \eqref{firsteq}-\eqref{1DIC} with all sources set to zero except $\psi(t)$. Then \Cref{ass:strong_assumption} does not hold.
\end{theorem}
\begin{proof}
 We will proceed with this proof by first finding $G\psi=(u,p)$ when $\psi(t)\in C_0^1(0,T)$. We want to show that there exists $(\varphi_m)_{m \in \mathbb{N}}$ such that 
\[\frac{(G\varphi_m, G\varphi_m)}{\|\varphi_m\|_{L^1(0,T)}^2} =\frac{\|p_m\|^2_{L^2(0,T)}+\|u_m\|^2_{L^2(0,T)}}{\|\varphi_m\|_{L^1(0,T)}^2} \to \infty~~\text{as }m \to \infty.\] 
%However, it is enough to show that 
%\[\frac{\|p_m\|^2}{\|\varphi_m\|_{L^1(0,T)}^2} \to \infty~~\text{as }m \to \infty.\]

Plugging \eqref{pequal} into \eqref{firsteq}-\eqref{1DIC}, we see that $u(x,t)$ needs to satisfy
\begin{align*}
    \frac{\partial^2u}{\partial x\partial t}-kH_v\frac{\partial^4u}{\partial x^3\partial t}-kH_e\frac{\partial^3 u}{\partial x^3}&=0 \quad & \forall (x,t) \in [0,L) \times [0,T]\\
    u(0,t)&=0 \quad &\forall t \in [0,T]\\
    H_v \frac{\partial^2 u}{\partial x\partial t} (0,t)+ H_e \frac{\partial u}{\partial x}(0,t)&=0 \quad &\forall t \in [0,T]\\
    -kH_v \frac{\partial^3 u}{\partial x^2 \partial t}(L,t) -kH_e \frac{\partial^2 u}{\partial x^2}(L,t)&=\psi(t) \quad &\forall t \in [0,T]\\
\frac{\partial u}{\partial x}(x,0)&=0 \quad &\forall x \in [0,L]
\end{align*}
Let $y(x,t)=u(x,t)+\frac{x^2}{2L}\Psi(t)$ where 
\[\Psi(t)=\frac{1}{kH_v}e^{-H_e/H_vt}\int_0^t \psi(\xi) e^{H_e/H_v \xi}d\xi.\]
Then $y(x,t)$ satisfies
\begin{align*}
    \frac{\partial^2y}{\partial x\partial t}-kH_v\frac{\partial^4y}{\partial x^3\partial t}-kH_e\frac{\partial^3 y}{\partial x^3}&=\frac{x}{L}\Psi'(t)\quad & \forall (x,t) \in [0,L] \times [0,T]\\
    y(0,t)&=0 \quad &\forall t \in [0,T]\\
    H_v \frac{\partial^2 y}{\partial x\partial t} (0,t)+ H_e \frac{\partial y}{\partial x}(0,t)&=0 \quad &\forall t \in [0,T]\\
    -kH_v \frac{\partial^3 y}{\partial x^2 \partial t}(L,t) -kH_e \frac{\partial^2 y}{\partial x^2}(L,t)&=0 \quad &\forall t \in [0,T]\\
\frac{\partial y}{\partial x}(x,0)&=0 \quad &\forall x \in [0,L]
\end{align*}
Let $y(x,t)=\sum_{n=1}^\infty f_n(t) \sqrt{2/L}(1-\cos(\lambda_n x))$ where $\lambda_n=\frac{(2n-1)\pi}{2L}$. Notice that $y(0,t)=0$, $\frac{\partial y}{\partial x}(0,t)=0$, and $\frac{\partial^2 y}{\partial x^2}(L,t)=0$. By plugging $y(x,t)$ into the first line of the PDE, we see that
\begin{align*} &\sum_{n=1}^\infty f'_n(t)\lambda_n\sqrt{2/L}\sin(\lambda_n x) + \lambda_n^3 kH_v \sum_{n=1}^\infty f'_n(t) \sqrt{2/L} \sin(\lambda_n x)+\lambda_n^3kH_e \sum_{n=1}^\infty f_n(t)\sqrt{2/L}\sin(\lambda_n x)=\frac{x}{L} \Psi'(t)
\end{align*}
and 
\[\sum_{n=1}^\infty f_n(0)\sqrt{2/L} \sin(\lambda_n x)=0.\]
Note that $\left(\sqrt{2/L}\sin(\lambda_n x)\right)_{n \in \mathbb{N}}$ is a complete orthonormal basis. Hence, multiplying both sides of these equations by $\sqrt{2/L}\sin(\lambda_n x)$, integrating these equations from $0$ to $L$, setting $c_n=(\frac{x}{L}, \sqrt{2/L}\sin(\lambda_n x))_{L^2(0,L)}=4\sqrt{2L}(-1)^{n+1}/((2n-1)^2\pi^2)$ we have
\[(\lambda_n+ \lambda_n^3 k H_v)f_n'(t)+kH_e \lambda_n^3 f_n(t)=c_n\Psi'(t)\quad \text{and} \quad f_n(0)=0.\]
Therefore, $f_n(t)=\frac{1}{\lambda_n+\lambda_n^3kH_v}e^{-kH_e\lambda_n^2\gamma_nt}\int_0^t  e^{(kH_e\lambda_n^2\gamma_n\xi)}c_n\Psi'(\xi) d \xi $ where $\gamma_n=\frac{1}{1+\lambda_n^2 kH_v}$. Hence,
\begin{align*}y(x,t)&=\sum_{n=1}^\infty \frac{1}{\lambda_n+\lambda_n^3kH_v} e^{-kH_e\lambda_n^2\gamma_nt}\int_0^t  e^{kH_e\lambda_n^2\gamma_n\xi}c_n\Psi'(\xi) d \xi \sqrt{2/L}(1-\cos(\lambda_n x))\\
u(x,t)&=\sum_{n=1}^\infty \frac{\gamma_n}{\lambda_n} e^{-kH_e\lambda_n^2\gamma_nt}\int_0^t e^{kH_e\lambda_n^2\gamma_n\xi}c_n\Psi'(\xi) d \xi \sqrt{2/L}(1-\cos(\lambda_n x))-\frac{x^2}{2L}\Psi(t)\end{align*}
and from \eqref{pequal},
\begin{align*}
&p(x,t)= -H_e \frac{x}{L} \Psi(t)-H_v\frac{x}{L}\Psi'(t)+\sum_{n=1}^\infty c_n\lambda_n \sqrt{2/L}\sin(\lambda_n x) \\
 \MoveEqLeft[-1]\left( -kH_eH_v \lambda_n\gamma_n^2e^{-kH_e\gamma_n\lambda_n^2t}\int_0^t e^{kH_e\lambda_n^2\gamma_n \xi}\Psi'(\xi) d \xi+ \frac{H_v\gamma_n}{\lambda_n}\Psi'(t)+\frac{H_e\gamma_n}{\lambda_n}e^{-kH_e\lambda_n^2\gamma_nt}\int_0^t e^{kH_e\lambda_n^2\gamma_n\xi}\Psi'(\xi) d\xi\right)\\
&=-\frac{x}{L}(H_e \Psi(t) +H_v \Psi'(t)) + \sum_{n=1}^\infty c_n \sqrt{2/L} \lambda_n \sin(\lambda_nx) \left(\frac{H_e\gamma_n^2}{\lambda_n}e^{-kH_e\lambda_n^2\gamma_n t}\int_0^t e^{kH_e\lambda_n^2 \gamma_n\xi}\Psi'(\xi) d\xi + \frac{H_v\gamma_n}{\lambda_n} \Psi'(t)\right).
\end{align*}
We notice 
\begin{equation} \label{genPsi'} \Psi'(t)=\frac{1}{kH_v}\left(\frac{-H_e}{H_v}e^{-H_e/H_vt}\int_0^t \psi(\xi) e^{H_e/H_v \xi}d\xi+\psi(t)\right).\end{equation}
Hence,
\begin{align*}
p(x,t)&=-\frac{x}{Lk}\psi(t)+\sum_{n=1}^\infty c_n \sqrt{2/L} \sin(\lambda_nx) \left(H_e\gamma_n^2e^{-kH_e\lambda_n^2\gamma_n t}\int_0^t e^{kH_e\lambda_n^2\gamma_n \xi}\Psi'(\xi) d\xi 
+ H_v\gamma_n \Psi'(t)\right).
\end{align*}
Therefore, 
\begin{align*}&\|p\|_{L^2(0,T;L^2(\Omega))}^2 =\int_0^T \sum_{n=1}^\infty c_n^2 \left(H_e\gamma_n^2e^{-kH_e\lambda_n^2\gamma_n t}\int_0^t e^{kH_e\lambda_n^2 \gamma_n \xi}\Psi'(\xi) d\xi 
+ H_v\gamma_n\Psi'(t)\right)^2 \\
 \MoveEqLeft[-1]-2\frac{c_n^2}{k}\psi(t)\left(H_e\gamma_n^2e^{-kH_e\lambda_n^2 \gamma_n t}\int_0^t e^{kH_e\lambda_n^2 \gamma_n \xi} \Psi'(\xi) d \xi + H_v \gamma_n\Psi'(t)\right) + \frac{c_n^2}{k^2}\psi^2(t)dt. 
\end{align*}
Recalling \eqref{genPsi'}, we see that when $\psi=\varphi_m=e^{mt}$, we have 
\begin{equation}\begin{split} \label{Psi'}
    \Psi'(t)&=\frac{1}{kH_v}\left(-\frac{H_e}{H_v} e^{-H_e/H_vt}\int_0^t e^{m\xi}e^{H_e/H_v\xi}d\xi +e^{mt}\right)\\
    &=\frac{1}{kH_v}\left(-\frac{H_e}{H_v}\left(\frac{e^{mt}}{m+H_e/H_v}-\frac{e^{-H_e/H_v t}}{m+H_e/H_v}\right)+ e^{mt}\right)\\
    &=\frac{1}{kH_v}\left(\frac{H_e e^{-H_e/H_vt}}{mH_v+H_e}-\frac{H_e e^{mt}}{mH_v  + H_e}+e^{mt} \right)\\
    &=\frac{1}{kH_v} \left(\frac{H_e e^{-H_e/H_vt}}{mH_v+H_e}+\frac{H_vm e^{mt}}{H_v m + H_e}\right) \geq 0
\end{split} \end{equation}
Hence,
\begin{align*}
   &e^{-kH_e \lambda_n^2\gamma_n t}\int_0^te^{kH_e \lambda_n^2\gamma_n \xi}\Psi'(\xi)d \xi\\
   &=\frac{1}{kH_v(mH_v+ H_e)}\left(\frac{H_e}{kH_e \lambda_n^2\gamma_n-H_e/H_v}\left(e^{-H_e/H_vt}-e^{-kH_e \lambda_n^2\gamma_n t}\right) + \frac{H_vm}{kH_e\lambda_n^2\gamma_n +m}\left(e^{mt}-e^{-kH_e\lambda_n^2 \gamma_n t} \right)\right)\\
   &=\frac{1}{kH_v(mH_v+ H_e)}\left(\frac{H_v}{k \lambda_n^2\gamma_nH_v-1}\left(e^{-H_e/H_vt}-e^{-kH_e \lambda_n^2\gamma_n t}\right) + \frac{H_vm}{kH_e\lambda_n^2\gamma_n +m}\left(e^{mt}-e^{-kH_e\lambda_n^2 \gamma_n t} \right)\right)\\
\end{align*}
Recall
$\gamma_n=\frac{1}{1+\lambda_n^2k H_v}$, so
\begin{align*}
   k\lambda_n^2 \gamma_n H_v-1&=\frac{k\lambda_n^2H_v}{1+ \lambda_n^2kH_v}-\frac{1+\lambda_n^2kH_v}{1+\lambda_n^2kH_v}\\
   &=\frac{-1}{1+\lambda_n^2kH_v}=-\gamma_n
   \end{align*}
and 
\[-kH_e \lambda_n^2 \gamma_n=-\frac{kH_e \lambda_n}{1+\lambda_n^2 kH_v}>-\frac{kH_e\lambda_n^2}{\lambda_n^2kH_v}=-\frac{H_e}{H_v},\] which implies
$1
\geq e^{-kH_e\lambda_n^2\gamma_n t}-e^{-H_e/H_vt}\geq 0$ for $t \in [0,T]$.
Hence,
\begin{equation} \label{ekintek} e^{-kH_e\lambda_n^2\gamma_nt}\int_0^t e^{kH_e\lambda_n^2\gamma_n \xi}\Psi'(\xi) d \xi\geq \frac{1}{kH_v(mH_v+ H_e)}\left(\frac{H_v}{\gamma_n} + \frac{H_vm}{kH_e\lambda_n^2\gamma_n +m}\left(e^{mt}-e^{-kH_e\lambda_n^2 \gamma_n t} \right)\right)\end{equation}
Let $l>0$ satisfy
\begin{equation} \label{gammal}\lambda_l^2=\frac{(2l-1)^2\pi^2}{4L^2}> \frac{8}{k H_v}.~~\text{Then }\gamma_l=\frac{1}{1+ \lambda_l^2kH_v}<\frac{1}{\lambda_l^2kH_v} <\frac{k}{8}.\end{equation}
%Notice that since $\gamma_n=\frac{1}{1+\lambda_n^2k H_v}$
%\begin{align*}
%   k\lambda_n^2 \gamma_n H_v-1&=\frac{k\lambda_n^2H_v}{1+ \lambda_n^2kH_v}-\frac{1+\lambda_n^2kH_v}{1+\lambda_n^2kH_v}\\
%   &=\frac{-1}{1+\lambda_n^2kH_v}=-\gamma_n
%\end{align*}
%We see that 
%\begin{align*}
%    &e^{-kH_e \lambda_n^2\gamma_n t}\int_0^te^{kH_e \lambda_n^2\gamma_n \xi}\Psi'(\xi)d \xi\\
%    &=\frac{1}{kH_ve^{mT}(mH_v+H_e)}\left(\frac{H_v}{\gamma_n}\left(e^{-kH_e\lambda_n^2 \gamma_n t}-e^{-H_e/H_vt}\right)+ \frac{H_v}{kH_e\lambda_n^2 \gamma_n +m}\left(e^{mt}-e^{-kH_e\lambda_n^2 \gamma_n t}\right)\right)
%\end{align*}
%Notice $e^{mt}-e^{-kH_e\lambda_n^2 \gamma_n t}\geq 0$ for all $t\in [0,T]$. Since 
%\[-kH_e \lambda_n^2 \gamma_n=-\frac{kH_e \lambda_n}{1+\lambda_n^2 kH_v}>-\frac{kH_e\lambda_n^2}{\lambda_n^2kH_v}=-\frac{H_e}{H_v},\]
%$e^{-kH_e\lambda_n^2\gamma_n t}-e^{-H_e/H_vt}\geq 0$ for $t \in [0,T]$. We conclude $e^{-kH_e\lambda_n^2 \gamma_n t} \int_0^t e^{kH_e \lambda_n^2\gamma_n \xi}\Psi'(\xi) \geq 0$ for all $t \in [0,T]$. Hence,
%\begin{align*}
%    &\|p_m\|_{L^2(0,T;L^2(\Omega))}^2 \geq \int_0^T \sum_{n=1}^\infty c_n^2\frac{H_v^2}{(\lambda_n + \lambda_n^3kH_v)^2(kH_ve^{mT})^2}\left(\frac{H_e e^{-H_e/H_vt}}{mH_v+H_e}+\frac{H_vm e^{mt}}{mH_v  + H_e} \right)^2\\
%    &\geq \int_0^T c_1^2 \frac{H_v^4 m^2 e^{2mt}}{(\lambda_n+ \lambda_n^3kH_v)^2k^2H_v^2(mH_v+H_e)^2} dt\\
%    &\geq \int_0^T c_1^2 \frac{H_v^4e^{2mt}}{(\lambda_1+ \lambda_1^3kH_v)^2k^2 H_v^2(H_v+H_e)^2}\\
%    &=\frac{\alpha(e^{2mT}-1)}{2m}
%\end{align*}
Notice that 
    \begin{equation} \nonumber \begin{split}\|p_m\|_{L^2(0,T;L^2(\Omega))}^2&\geq \int_0^T c_l^2 \left(H_e\gamma_l^2e^{-kH_e\lambda_l^2\gamma_l t}\int_0^t e^{kH_e\lambda_l^2 \gamma_l \xi}\Psi'(\xi) d\xi 
+ H_v\gamma_l\Psi'(t)\right)^2 \\
 \MoveEqLeft[-1]-2\frac{c_l^2}{k}\psi(t)\left(H_e\gamma_l^2e^{-kH_e\lambda_l^2 \gamma_l t}\int_0^t e^{kH_e\lambda_l^2 \gamma_l \xi} \Psi'(\xi) d \xi + H_v \gamma_l\Psi'(t)\right) + \frac{c_l^2}{k^2}\psi^2(t)dt \end{split} \end{equation}
Dropping the first term since it is positive and applying \eqref{ekintek} and \eqref{Psi'} we see
\begin{equation}
    \nonumber
    \begin{split}
        \|p_m\|_{L^2(0,T;L^2(\Omega))}^2&\geq \int_0^T-2\frac{c_l^2}{k} e^{mt}\frac{H_e \gamma_l^2}{kH_v(mH_v+H_e)}\left(\frac{H_v}{\gamma_l} + \frac{H_vm}{kH_e \lambda_l^2 \gamma_l+m}\left(e^{mt}-e^{-kH_e \lambda_l^2 \gamma_l t}\right) \right)\\
         \MoveEqLeft[-1] -2\frac{c_l^2}{k} e^{mt}\frac{\gamma_l}{k}\left(\frac{H_e e^{-H_e/H_v t}}{mH_v+H_e}+\frac{H_v me^{mt}}{H_v m + H_e}\right)+ \frac{c_l^2}{k^2} e^{2mt}dt
    \end{split}
\end{equation}
Recalling that $-(e^{mt}-e^{-kH_e \lambda_l^2 \gamma_l t})\geq -e^{mt}$, $-\left(\frac{H_e e^{-H_e/H_v t}}{mH_v+H+e}+\frac{H_v m e^{mt}}{H_vm+H_e}\right) \geq -(1+ e^{mt})$, and $mH_v+H_e \geq H_e$ we see
\begin{equation} \nonumber \begin{split} \|p_m\|_{L^2(0,T;L^2(\Omega))}^2 &\geq \int_0^T-2\frac{c_l^2}{k} e^{mt}\frac{H_e \gamma_l^2}{kH_v(mH_v+H_e)} \left(\frac{H_v}{\gamma_l}+\frac{H_vm}{ m}e^{mt}\right)-2\frac{c_l^2}{k}e^{mt}\frac{\gamma_l}{k}(1+ e^{mt})+ \frac{c_l^2}{k^2} e^{2mt}dt\\
&\geq \int_0^T-2\frac{c_l^2}{k} e^{mt}\left(\frac{H_e \gamma_l^2}{kH_vH_e} \frac{H_v}{\gamma_l}+\frac{H_e \gamma_l^2H_v}{kH_vH_e}e^{mt}\right)-2\frac{c_l^2}{k}e^{mt}\frac{\gamma_l}{k}(1+ e^{mt})+ \frac{c_l^2}{k^2} e^{2mt}dt\\
&\geq \int_0^T-2\frac{c_l^2}{k^2} e^{2mt} \left(\gamma_l+\gamma_l^2\right)-2\frac{c_l^2}{k^2} e^{2mt}2\gamma_l + \frac{c_l^2}{k^2} e^{2mt} dt\end{split}\end{equation}
Hence, applying \eqref{gammal}, we see
\begin{equation} \|p_m\|_{L^2(0,T;L^2(\Omega))}^2\geq \int_0^T-2\frac{c_l^2}{k^2} e^{2mt}\left(\frac{1}{8}+\frac{1}{64}+\frac{1}{4}\right)+\frac{c_l^2}{k^2} e^{2mt}dt\geq \int_0^T \frac{7c_l^2}{32^2}e^{2mt}dt=\frac{7c_l^2}{64k^2m}\left(e^{2mT}-1\right). \end{equation}
Also, note $\|\varphi_m\|_{L^1(0,T)}^2= \frac{(e^{mt}-1)^2}{m^2}$. Therefore,
\begin{align*}
\frac{\nabla ^2 j(w)(\varphi_m, \varphi_m)}{\|\varphi_m\|_{L^1(0,T)}^2}&=\frac{(G\varphi_m, G\varphi_m)}{\|\varphi_m\|_{L^1(0,T)}^2} =\frac{\|p_m\|^2_{L^2(0,T;L^2(0,L))}+\|u_m\|^2_{L^2(0,T;L^2(0,L))}}{\|\varphi_m\|_{L^1(0,T)}^2} \geq \frac{\|p_m\|^2_{L^2(0,T;L^2(0,L))}}{\|\varphi_m\|_{L^1(0,T)}^2}\\
&\geq \frac{\frac{7c_l}{64mk^2}(e^{2mT}-1)m^2}{(e^{mT}-1)^2} \to \infty ~~\text{as }m \to \infty.
\end{align*}
Therefore, \Cref{ass:strong_assumption} is not satisfied.
\end{proof}
We now consider the poroviscoelastic case where the control is used as the time component of the source $S$, and show \Cref{ass:strong_assumption} is not satisfied.
\begin{theorem}
Let $\lambda=0$ and $\delta>0$. Let $k(x,t)=k$ and $\chi(x)\in L^2(0,L)$. Define $G:L^2(0,T)\to L^2(0,T;L^2(0,L) \times L^2(0,L))$ to be the map that maps $s(t)\in L^2(0,T)$ to the $(u,p)$ that satisfies \eqref{firsteq}-\eqref{1DIC} with all sources set to 0 except $S(x,t)=s(t)\chi(x)$. Then \Cref{ass:strong_assumption} does not hold. \end{theorem}
\begin{proof}
We will proceed with this proof by first computing $Gs=(u,p)$ and then showing a lower estimate on $\|p\|_{L^2(0,T;L^2(\Omega))}^2$.
Plugging \eqref{pequal} into the PDE we see $u(x,t)$ needs to satisfy
\begin{align*}
    \frac{\partial^2u}{\partial x\partial t}-kH_v\frac{\partial^4u}{\partial x^3\partial t}-kH_e\frac{\partial^3 u}{\partial x^3}&=S \quad & \forall (x,t) \in [0,L) \times [0,T]\\
    u(0,t)&=0 \quad &\forall t \in [0,T]\\
    H_v \frac{\partial^2 u}{\partial x\partial t} (0,t)+ H_e \frac{\partial u}{\partial x}(0,t)&=0 \quad &\forall t \in [0,T]\\
    -kH_v \frac{\partial^3 u}{\partial x^2 \partial t}(L,t) -kH_e \frac{\partial^2 u}{\partial x^2}(L,t)&=0 \quad &\forall t \in [0,T]\\
\frac{\partial u}{\partial x}(x,0)&=0 \quad &\forall x \in [0,L]
\end{align*}
Let $u(x,t)=\sum_{n=1}^\infty f_n(t) \sqrt{2/L}(1-\cos(\lambda_n x))$, where $\lambda_n=\frac{(2n-1)\pi}{2L}$. Notice that $u(0,t)=0$, $\frac{\partial u}{\partial x}(0,t)=0$, and $\frac{\partial^2 u}{\partial x^2}(L,t)=0$, so the boundary conditions are satisfied. We have
\[\sum_{n=1}^\infty f'_n(t)\lambda_n\sqrt{2/L}\sin(\lambda_n x) + \lambda_n^3 kH_v \sum_{n=1}^\infty f'_n(t) \sqrt{2/L} \sin(\lambda_n x)+\lambda_n^3kH_e \sum_{n=1}^\infty f_n(t)\sqrt{2/L}\sin(\lambda_n x) =s(t)\chi(x)\]
and 
\[\sum_{n=1}^\infty f_n(0)\sqrt{2/L} \sin(\lambda_n x)=0.\]
Note $\left(\sqrt{2/L}\sin(\lambda_n x)\right)_{n \in \mathbb{N}}$ is a complete orthonormal basis. Hence, multiplying both sides of these equations by $\sqrt{2/L}\sin(\lambda_n x)$, integrating these equations from $0$ to $L$, and setting $c_n=(\chi(x), \sqrt{2/L}\sin(\lambda_n x))_{L^2(0,L)}$, we have for all $n \in \mathbb{N}$
\[(\lambda_n+ \lambda_n^3 k H_v)f_n'(t)+kH_e \lambda_n^3 f_n(t)=c_n s(t)\quad \text{and} \quad f_n(0)=0.\]
Therefore, $f_n(t)=\frac{1}{\lambda_n+\lambda_n^3kH_v}e^{-kH_e\lambda_n^2t/(1+\lambda_n^2 kH_v)}\int_0^t  c_n e^{kH_e\lambda_n^2\xi/(1+\lambda_n^2 k H_v)}s(\xi) d \xi $. Hence,
\[u(x,t)=\sum_{n=1}^\infty \frac{c_n}{\lambda_n+\lambda_n^3kH_v} e^{-kH_e\lambda_n^2t/(1+\lambda_n^2 kH_v)}\int_0^t  e^{kH_e\lambda_n^2\xi/(1+\lambda_n^2 k H_v)}s(\xi) d \xi \sqrt{2/L}(1-\cos(\lambda_n x))\]
and \eqref{pequal} gives
\begin{align*}
p(x,t)&=\sum_{n=1}^\infty c_n\sqrt{2/L}\sin(\lambda_n x) \left( \frac{-kH_eH_v \lambda_n}{(1 + \lambda_n^2kH_v)^2}e^{-kH_e\lambda_n^2t/(1+\lambda_n^2 kH_v)}\int_0^t e^{kH_e\lambda_n^2\xi/(1+\lambda_n^2 k H_v)}s(\xi) d \xi \right.\\
 \MoveEqLeft[-1]\left.+ \frac{H_v}{\lambda_n+ \lambda_n^3kH_v}s(t)+\frac{H_e}{\lambda_n+ \lambda_n^3 kH_v}e^{-kH_e \lambda_n^2t/(1+\lambda_n^2 kH_v)}\int_0^t e^{kH_e \lambda_n^2 \xi/(1+ \lambda_n^2 k H_v)} s(\xi) d\xi\right)\\
&=\sum_{n=1}^\infty c_n \sqrt{2/L}\sin(\lambda_n x) \left(\left(\frac{-kH_eH_v\lambda_n^2}{\lambda_n(1+ \lambda_n^2kH_v)^2}+\frac{H_e(1+ \lambda_n^2k H_v)}{\lambda_n(1+ \lambda_n^2kH_v)^2}  \right)e^{-kH_e\lambda_n^2t/(1+ \lambda_n^2kH_v)}\right.\\
 \MoveEqLeft[-1]\left.\int_0^t e^{kH_e\lambda_n^2\xi/(1+ \lambda_n^2 k H_v)} s(\xi) d \xi + \frac{H_v}{\lambda_n + \lambda_n^3kH_v}s(t)\right)\\
&=\sum_{n=1}^\infty c_n \sqrt{2/L}\sin(\lambda_n x) \left(\frac{H_e}{\lambda_n(1+\lambda_n^2 kH_v)^2}e^{-kH_e\lambda_n^2t/(1+ \lambda_n^2kH_v)}\int_0^t e^{kH_e\lambda_n^2\xi/(1+ \lambda_n^2 k H_v)} s(\xi) d \xi \right.\\
 \MoveEqLeft[-1]\left.+ \frac{H_v}{\lambda_n + \lambda_n^3kH_v}s(t)\right).\\
\end{align*}
Therefore, when $s(t)$ is strictly non-negative,
\[\|p\|_{L^2(0,T;L^2(\Omega))}^2 \geq \sum_{n=1}^\infty \frac{c_n^2 H_v^2}{(\lambda_n+ \lambda_n^3kH_v)^2} \|s\|_{L^2(0,T)}^2\geq C \|s\|^2_{L^2(0,T)}\]
for some $C>0$.
Hence, using $\varphi_m=m\left(\frac{t}{T}\right)^m$ (as was done in the proof of \Cref{prp:strong_assumption_not_satisfied}), we have
\begin{align*}
\frac{\nabla ^2 j(w)(\varphi_m, \varphi_m)}{\|\varphi_m\|_{L^1(0,T)}^2}=\frac{(G\varphi_m,G \varphi_m)}{\|\varphi_m\|_{L^1(0,T)}^2} \geq \frac{\|p_m\|_{L^2(0,T;L^2(\Omega))}}{\|\varphi_m\|_{L^1(0,T)}^2} \geq \frac{C\|\varphi_m\|_{L^2(0,T)}}{\|\varphi_m\|_{L^1(0,T)}^2} \to \infty~~\text{as } m \to \infty.
\end{align*}
Therefore, \Cref{ass:strong_assumption} is not satisfied. 
\end{proof}

\section{Computational Experiments}\label{sec:comp}
For our computational experiments, we use an instance of the one-dimensional
porous medium equations described in \Cref{sec:biots_model}.
We intend to analyze the effect of the mollification regularization
on the resulting instationarity and the objective values in practice.
Specifically, we consider the control input $\psi$, where 
\Cref{ass:strong_assumption} is satisfied for the poroelastic case
and violated for the poroviscoelastic case.
In \Cref{sec:setup}, the specific model as well as the discretization of the model and the 
trust-region subproblems are described as well as the details for the computational experiements and the homotopy.
We present and describe our results in 
\Cref{sec:results}.

\subsection{Numerical Experiments}\label{sec:setup}
We consider a discretized instance of \eqref{eq:p} that is governed by the PDE
introduced in \Cref{sec:biots_model} with a one-dimensional spatial domain
$\Omega = (0,2)$ and a one-dimensional time domain $(0,T) = (0,0.5)$.
We use the same discretization as in \cite{BGSW}. For the time horizon we use N=512
uniform intervals and an implicit Euler scheme. For the space discretization, we use a dual hybridized finite element method with 512 uniform intervals on each of which the control is constant with a
value in $W$. We choose the set of possible control realizations as $W=\{-7, -5, -3, -1, 0, 2\}$ in our computations.

Regarding the parameters of the PDE we use $k = 1$,
$\lambda_e = 1$, $\mu_e = 1$, $\mu_v = 0.25$,
$\lambda_v = 0.0774$. We execute the same experiments
for both the two choices $\delta = 0$ (poroelastic case)
and $\delta = 1$ (poroviscoelastic case).

Regarding the setup of the control problem, we choose the
structure given in \eqref{eq:control_problem}
where the PDE input choices are $\psi = w$ ($F = 0$, $S = 0$, $g = 0$) 
and $S = \chi w$ for a fixed function $\chi$ ($F = 0$, $g = 0$,
$\psi = 0$), where $w$ denotes the control function.
We choose the
penalty parameter value $\alpha = 5\cdot 10^{-5}$ to scale the
$\TV$-term in the objective. We run all
experiments for the choices $\lambda = 10^{-4}$,
$\lambda = 10^{-2}$, and $\lambda = 0$. We tabulate for
which of the settings \Cref{ass:strong_assumption} is violated or 
satisfied in \Cref{tbl:ass} according to the results obtained in 
\Cref{sec:biots_model}.
\begin{table}[h]
    \centering
    \caption{Satisfaction (T) and violation (F) of \Cref{ass:strong_assumption} for the
    different experiments.}
    \label{tbl:ass}
    \begin{tabular}{r|cccccc}
    \toprule
    &
    & Input is $S$ &
    & 
    & Input is $\psi$ &\\
    $\lambda = $ & $0$ & $10^{-4}$ & $10^{-2}$ & $0$ & $10^{-4}$ & $10^{-2}$ \\
    \midrule
    $\delta = 0$ & T & F & F & T & F & F \\
    $\delta = 1$ & F & F & F & F & F & F\\
    \bottomrule
    \end{tabular}
\end{table} We note that \Cref{ass:strong_assumption}
is not satisfied for both $\delta = 0$ and $\delta = 1$ 
for $\lambda = 10^{-2}$ and $\lambda = 10^{-4}$
by virtue of  \Cref{prp:strong_assumption_not_satisfied} regardless of the fact
which of the control inputs is chosen. Additionally, since \eqref{hessian} shows Assumption \ref{ass:weak_assumption} is satisfied in all cases, we note that Proposition \ref{prp:relax_assumption} implies Assumption \ref{ass:strong_assumption} is satisfied for $j_\epsilon$ in all the settings.

For the tracking terms, we choose 
$u_d(x,t) = 0.5 + (1 - t)^2\cos(50t)(-1.975x + 4)$
and $p_d(x,t) = 0.5 + \cos(50t)^2$ for all $(x,t) \in \Omega\times (0,T)$.
%We use a one-dimensional spatial domain that is discretized as
%in \cite{BGSW} into $512$ intervals. We discretize the PDE following
%the \cite{BGSW}. We choose $T = 0.5$ and discretize the time
%horizon uniformly into $N = 512$ intervals. $N = 512$ intervals are used
%for both the implicit Euler time-stepping scheme and
%and the discretization of the controls with interval-wise
%constant functions that assume the values in $W$ on the intervals.
The tracking-type term and the squared $L^2$-norm of the control,
are discretized using the trapezoidal rule for the same intervals.
The derivative of the of the first (reduced) term of the objective is
required to evaluate the linear part of the objective of the 
subproblem \eqref{eq:tr}. In order to compute the latter, we use a
first discretize, then optimize-based \cite{hinze2008optimization}
adjoint calculus.

In our executions of \Cref{alg:slip} on a computer, we select six feasible initial controls $w^0$, specifically
$w^0 \equiv \omega$ for all $\omega \in W$. Then, we replace the
infinite-dimensional trust-region subproblems with the discretizations
that are described above. We note that we have the (implicit) termination
criterion in \Cref{alg:slip} that the trust-region radius contracts to a
value below $T / N$ because we operate with limited precision and a
fixed discretization. In this case, the linear integer program that
arises after discretizing \eqref{eq:tr} has only one feasible point,
namely the function $d = 0$ with objective value $0$. Thus we always
run \Cref{alg:slip} until this situation occurs. The reset trust-region
radius is $\Delta^0 = 0.25 T$. The acceptance value for the ratio of
actual over predicted reduction is $\sigma = 10^{-3}$.

\Cref{alg:slip} is implemented in MATLAB. C++ is used for
	the subproblem solver implementation, which follows
	\cite{severitt2022efficient}. All computations were executed on a
	workstation with an AMD Epic 7742 CPU and 96 GB RAM.

For each of the 6 initializations $w^0$, we record the final control $w^f$,
the final objective value $j(w^f) + \alpha \TV(w^f)$, and the instationarity $C(w^f)$
on termination for all of these executions. Here,
$C(w^f) \coloneqq \|(\nabla j(w^f)(t_i))_{i=1}^{\# s}\|$,
where $t_1$, $\ldots$, $t_{\#s}$ for some $\#s \in \N$
denote the switching times of $w^f$, that is the
values $\hat{t}\in (0,T)$ such that
$\lim_{t \downarrow \hat{t}} w^f(t) \neq \lim_{t \uparrow \hat{t}} w^f(t)$.
Note again that the stationarity condition from \cref{dfn:l_stationary} becomes
$\nabla j(w^f)(t_i) = 0$ for all such switching times $t_1$, $\ldots$, $t_{\#s}$
if $\nabla j(w^f)$ is a continuous function, which is ensured by the regularity
of the solution of the adjoint equation. Consequently, the instationarity is the norm of
the vector of the $\#s$ individual violations of this instationarity condition.

Then we regularize $j$ by composing it with the application of a standard mollifier to the control
input following our recipe in \Cref{sec:input_regularization}, that is we replace $j$ by
$j \circ K_\varepsilon$ in \eqref{eq:p} and $\nabla j$ by $K_\varepsilon^*(\nabla j)\circ K_\varepsilon$
in \eqref{eq:tr}. For each of the six initial controls $w^0$ we execute a homotopy of $\varepsilon$
and \Cref{alg:slip} on the regularized problems with the following regularization parameter values
$$\varepsilon \in \{ \num{1.6e-2}, \num{8e-3}, \num{4e-3}, \num{2e-3}, \num{1e-3}, \num{5e-4}, \num{2.5e-4}, 0 \}.$$
We initialize \Cref{alg:slip} with $w^0$ for the largest regularization parameter value
$\varepsilon = \num{1.6e-2}$ and initialize the execution of \Cref{alg:slip} 
for the subsequent value of $\varepsilon$ with the final control function iterate of the
previous parameter value for $\varepsilon$. Again, we record the final controls, objective values,
and instationarities on termination.

For $\lambda > 0$, L-stationarity is not known to be a necessary optimality condition. Moreover, the
fixed discretization also implies that we cannot expect that
the final iterate is perfectly L-stationary even if 
\Cref{ass:strong_assumption} is satisfied. 
In order to provide a full picture, we have chosen to still
measure remaining instationarity for the final control iterates
and report the final instationarities for the unregularized
optimization and the homotopy but kindly ask the reader to
take these values with caution. We assess the remaining
instationarity of the final control iterate $w^f$ by 
evaluating
$C(w^f)$.

\subsection{Results}\label{sec:results}
The results achieved with the unregularized optimization differ significantly from those obtained
with the homotopy. We first note that for the same initialization, the 
run at the end of the homotopy with $\varepsilon = 0$ and a plain run of Algorithm \ref{alg:slip}
with $\varepsilon = 0$ with the initial guess from the beginning of the homotopy
produce different final iterates (approximately L-stationary points). Within
the homotopy, the initial guess for an execution of Algorithm \ref{alg:slip} is the final iterate of
the previous execution of Algorithm \ref{alg:slip} with a larger value of $\varepsilon$. Consequently,
since the initial guesses of the two runs with $\varepsilon = 0$ are different and the problem
is nonconvex, they can lead to different sequences of iterates that converge to different L-stationary
points.

\paragraph{Input choice $\psi$.}
We report the details of the numerical results when $\psi$ is used as the input choice in \Cref{tab:delta_0_first_experiment} for
the poroelastic case $\delta = 0$ and in \Cref{tab:first_experiment}
for the poroviscoelastic case. Detailed iteration numbers
over the different values of the homotopy are given in 
\Cref{tab:delta_0_iterations} for $\delta = 0$
and \Cref{tab:iterations} for $\delta = 1$.

\begin{table}[t]
    \centering
    \caption{Control function $\psi$, poroelastic case $\delta = 0$: Objective, instationarity, and (cumulative) outer iterations
    of \eqref{eq:p} for six initial controls achieved by
    executing \Cref{alg:slip} on the discretized problem as well
    as by executing a homotopy, abbreviated (H),
    of executions of \Cref{alg:slip}
    on regularized problems with regularization parameter driven to
    zero for the same initial controls.}
    \label{tab:delta_0_first_experiment}
    \begin{tabular}{lr|llllll}
    \toprule
    & Ini.
    & Final Obj.
    & Final Obj. (H)
    & Final Inst.
    & Final Inst. (H) 
    & Iter.
    & Cum.\ Iter.\ (H)
    \\
\midrule
$\lambda = 0$ &
1   &   \num{1.3474}   &   \num{1.3474}   &   \num{3.4e-06}   &   \num{2.2e-06}   &   35   &   43 \\
$\lambda = 10^{-4}$&
1   &   \num{1.3476}   &   \num{1.3476}   &   \num{3.0e-06}   &   \num{2.8e-06}   &   35   &   42 \\
$\lambda = 10^{-2}$&
1   &   \num{1.3616}   &   \num{1.3616}   &   \num{3.5e-06}   &   \num{2.2e-06}   &   31   &   47 \\
\midrule
$\lambda = 0$&
2   &   \num{1.3474}   &   \num{1.3474}   &   \num{2.6e-06}   &   \num{2.1e-06}   &   64   &   81 \\
$\lambda = 10^{-4}$&
2   &   \num{1.3476}   &   \num{1.3476}   &   \num{3.4e-06}   &   \num{2.9e-06}   &   47   &   79 \\
$\lambda = 10^{-2}$&
2   &   \num{1.3616}   &   \num{1.3616}   &   \num{7.2e-06}   &   \num{3.8e-06}   &   33   &   56 \\
\midrule
$\lambda = 0$&
3   &   \num{1.3474}   &   \num{1.3474}   &   \num{3.8e-06}   &   \num{2.3e-06}   &   33   &   78 \\
$\lambda = 10^{-4}$&
3   &   \num{1.3476}   &   \num{1.3476}   &   \num{2.1e-06}   &   \num{2.1e-06}   &   38   &   84 \\
$\lambda = 10^{-2}$&
3   &   \num{1.3616}   &   \num{1.3616}   &   \num{9.0e-06}   &   \num{3.6e-06}   &   36   &   67 \\
\midrule
$\lambda = 0$&
4   &   \num{1.3475}   &   \num{1.3474}   &   \num{3.0e-06}   &   \num{1.4e-06}   &   22   &   48 \\
$\lambda = 10^{-4}$&
4   &   \num{1.3477}   &   \num{1.3476}   &   \num{3.4e-06}   &   \num{2.9e-06}   &   22   &   43 \\
$\lambda = 10^{-2}$&
4   &   \num{1.3616}   &   \num{1.3615}   &   \num{2.7e-06}   &   \num{3.4e-06}   &   26   &   35 \\
\midrule
$\lambda = 0$&
5   &   \num{1.3474}   &   \num{1.3474}   &   \num{4.3e-06}   &   \num{2.5e-06}   &   21   &   56 \\
$\lambda = 10^{-4}$&
5   &   \num{1.3476}   &   \num{1.3476}   &   \num{2.3e-06}   &   \num{3.0e-06}   &   26   &   71 \\
$\lambda = 10^{-2}$&
5   &   \num{1.3616}   &   \num{1.3616}   &   \num{4.5e-06}   &   \num{2.1e-06}   &   27   &   55 \\
\midrule
$\lambda = 0$&
6   &   \num{1.3474}   &   \num{1.3474}   &   \num{1.9e-06}   &   \num{1.3e-06}   &   50   &   90 \\
$\lambda = 10^{-4}$&
6   &   \num{1.3476}   &   \num{1.3476}   &   \num{2.5e-06}   &   \num{2.2e-06}   &   49   &   69 \\
$\lambda = 10^{-2}$&
6   &   \num{1.3616}   &   \num{1.3617}   &   \num{8.3e-06}   &   \num{5.4e-06}   &   39   &   70 \\
    \bottomrule    
    \end{tabular}
\end{table}
\begin{table}[t]
    \caption{Control function $\psi$, poroviscoelastic case $\delta = 1$: Objective, instationarity, and (cumulative) outer iterations
    of \eqref{eq:p} for six initial controls achieved by
    executing \Cref{alg:slip} on the discretized problem as well
    as by executing a homotopy, abbreviated (H),
    of executions of \Cref{alg:slip}
    on regularized problems with regularization parameter driven to
    zero for the same initial controls.}
    \label{tab:first_experiment}
    \begin{tabular}{lr|llllll}
    \toprule
    & Ini.
    & Final Obj.
    & Final Obj. (H)
    & Final Inst.
    & Final Inst. (H) 
    & Iter.
    & Cum.\ Iter.\ (H)
    \\
    \midrule
$\lambda = 0$ &
1   &   \num{1.3647}   &   \num{1.3625}   &   \num{1.7e-04}   &   \num{4.0e-05}   &   48   &   265 \\
$\lambda = 10^{-4}$ &
1   &   \num{1.3647}   &   \num{1.3625}   &   \num{1.7e-04}   &   \num{1.1e-06}   &   40   &   232 \\
$\lambda = 10^{-2}$ &
1   &   \num{1.3732}   &   \num{1.3697}   &   \num{2.4e-04}   &   \num{5.8e-05}   &   35   &   240 \\
\midrule
$\lambda = 0$ &
2   &   \num{1.3646}   &   \num{1.3624}   &   \num{1.1e-04}   &   \num{1.3e-06}   &   62   &   289 \\
$\lambda = 10^{-4}$ &
2   &   \num{1.3644}   &   \num{1.3625}   &   \num{1.1e-04}   &   \num{1.1e-06}   &   57   &   278 \\
$\lambda = 10^{-2}$ &
2   &   \num{1.3715}   &   \num{1.3696}   &   \num{9.2e-05}   &   \num{1.8e-06}   &   67   &   349 \\
\midrule
$\lambda = 0$ &
3   &   \num{1.3657}   &   \num{1.3624}   &   \num{1.1e-04}   &   \num{1.3e-06}   &   26   &   268 \\
$\lambda = 10^{-4}$ &
3   &   \num{1.3656}   &   \num{1.3625}   &   \num{1.1e-04}   &   \num{1.1e-06}   &   31   &   262 \\
$\lambda = 10^{-2}$ &
3   &   \num{1.3730}   &   \num{1.3696}   &   \num{1.1e-04}   &   \num{1.8e-06}   &   26   &   241 \\
\midrule
$\lambda = 0$ &
4   &   \num{1.3624}   &   \num{1.3624}   &   \num{1.3e-06}   &   \num{1.3e-06}   &   24   &   201 \\
$\lambda = 10^{-4}$ &
4   &   \num{1.3625}   &   \num{1.3625}   &   \num{1.1e-06}   &   \num{1.1e-06}   &   22   &   207 \\
$\lambda = 10^{-2}$ &
4   &   \num{1.3696}   &   \num{1.3696}   &   \num{1.8e-06}   &   \num{1.8e-06}   &   22   &   194 \\
\midrule
$\lambda = 0$ &
5   &   \num{1.3652}   &   \num{1.3624}   &   \num{7.5e-05}   &   \num{1.3e-06}   &   34   &   240 \\
$\lambda = 10^{-4}$ &
5   &   \num{1.3652}   &   \num{1.3625}   &   \num{7.5e-05}   &   \num{1.1e-06}   &   34   &   264 \\
$\lambda = 10^{-2}$ &
5   &   \num{1.3716}   &   \num{1.3696}   &   \num{6.7e-05}   &   \num{1.8e-06}   &   41   &   283 \\
\midrule
$\lambda = 0$ &
6   &   \num{1.3634}   &   \num{1.3624}   &   \num{1.1e-04}   &   \num{1.3e-06}   &   43   &   303 \\
$\lambda = 10^{-4}$ &
6   &   \num{1.3635}   &   \num{1.3625}   &   \num{1.1e-04}   &   \num{1.1e-06}   &   43   &   312 \\
$\lambda = 10^{-2}$ &
6   &   \num{1.3700}   &   \num{1.3696}   &   \num{6.3e-05}   &   \num{1.8e-06}   &   64   &   349 \\
    \bottomrule    
    \end{tabular}
\end{table}

\begin{table}[!ht]
    \centering
    \caption{Control function $\psi$, poroelastic case $\delta = 0$: Number of iterations required by the executions of 
    \Cref{alg:slip} for the different initializations of the computational example over the different
    values of $\varepsilon$ of the homotopy, cumulative for the homotopy (H), and for the unregularized problem (U).}
    \label{tab:delta_0_iterations}
    ~
    \begin{tabular}{lr|llllllllll}
    \toprule
    && $\varepsilon = $ \\
    & Ini.
    & ${0.016}$
    & ${0.008}$
    & ${0.004}$
    & ${0.002}$
    & ${0.001}$
    & ${0.0005}$
    & ${0.00025}$
    & ${0}$
    & (H)
    & (U)
    \\
\midrule
$\lambda = 0$ &
1   &   30   &   4   &   2   &   1   &   2   &   1   &   1   &   2   &   43   &   35 \\
$\lambda = 10^{-4}$ &
1   &   31   &   2   &   2   &   2   &   1   &   1   &   1   &   2   &   42   &   35 \\
$\lambda = 10^{-2}$ &
1   &   28   &   8   &   4   &   2   &   1   &   1   &   1   &   2   &   47   &   31 \\
\midrule
$\lambda = 0$ &
2   &   53   &   13   &   5   &   3   &   2   &   1   &   1   &   3   &   81   &   64 \\
$\lambda = 10^{-4}$ &
2   &   57   &   5   &   7   &   5   &   2   &   1   &   1   &   1   &   79   &   47 \\
$\lambda = 10^{-2}$ &
2   &   35   &   4   &   7   &   3   &   2   &   1   &   1   &   3   &   56   &   33 \\
\midrule
$\lambda = 0$ &
3   &   66   &   2   &   3   &   2   &   1   &   1   &   1   &   2   &   78   &   33 \\
$\lambda = 10^{-4}$ &
3   &   60   &   2   &   12   &   3   &   2   &   1   &   1   &   3   &   84   &   38 \\
$\lambda = 10^{-2}$ &
3   &   47   &   7   &   6   &   1   &   2   &   1   &   1   &   2   &   67   &   36 \\
\midrule
$\lambda = 0$ &
4   &   29   &   8   &   2   &   2   &   3   &   1   &   1   &   2   &   48   &   22 \\
$\lambda = 10^{-4}$ &
4   &   23   &   4   &   6   &   3   &   4   &   1   &   1   &   1   &   43   &   22 \\
$\lambda = 10^{-2}$ &
4   &   22   &   2   &   3   &   3   &   2   &   1   &   1   &   1   &   35   &   26 \\
\midrule
$\lambda = 0$ &
5   &   41   &   2   &   2   &   6   &   1   &   1   &   1   &   2   &   56   &   21 \\
$\lambda = 10^{-4}$ &
5   &   49   &   2   &   2   &   4   &   8   &   1   &   1   &   4   &   71   &   26 \\
$\lambda = 10^{-2}$ &
5   &   41   &   2   &   3   &   3   &   1   &   1   &   1   &   3   &   55   &   27 \\
\midrule
$\lambda = 0$ &
6   &   79   &   2   &   2   &   2   &   1   &   1   &   1   &   2   &   90   &   50 \\
$\lambda = 10^{-4}$ &
6   &   58   &   2   &   2   &   2   &   1   &   1   &   1   &   2   &   69   &   49 \\
$\lambda = 10^{-2}$ &
6   &   50   &   6   &   6   &   1   &   2   &   1   &   1   &   3   &   70   &   39 \\  
\bottomrule    
    \end{tabular}
\end{table}

\begin{table}[!ht]
    \centering
    \caption{Control function $\psi$, poroviscoelastic case $\delta = 1$: Number of iterations required by the executions of 
    \Cref{alg:slip} for the different initializations of the computational example over the different
    values of $\varepsilon$ of the homotopy, cumulative for the homotopy (H), and for the unregularized problem (U).}
    \label{tab:iterations}
    \begin{tabular}{lr|llllllllll}
    \toprule
    & $\varepsilon = $ \\
    & Ini.
    & ${0.016}$
    & ${0.008}$
    & ${0.004}$
    & ${0.002}$
    & ${0.001}$
    & ${0.0005}$
    & ${0.00025}$
    & ${0}$
    & (H)
    & (U)
    \\
\midrule
$\lambda = 0$ &
1   &   120   &   59   &   44   &   23   &   11   &   1   &   1   &   6   &   265   &   48 \\
$\lambda = 10^{-4}$ &
1   &   112   &   40   &   35   &   26   &   8   &   1   &   1   &   9   &   232   &   40 \\
$\lambda = 10^{-2}$ &
1   &   102   &   58   &   43   &   19   &   11   &   1   &   1   &   5   &   240   &   35 \\
\midrule
$\lambda = 0$ &
2   &   145   &   62   &   47   &   24   &   6   &   1   &   1   &   3   &   289   &   62 \\
$\lambda = 10^{-4}$ &
2   &   134   &   64   &   38   &   15   &   17   &   1   &   1   &   8   &   278   &   57 \\
$\lambda = 10^{-2}$ &
2   &   194   &   61   &   43   &   30   &   13   &   1   &   1   &   6   &   349   &   67 \\
\midrule
$\lambda = 0$ &
3   &   99   &   56   &   62   &   31   &   10   &   1   &   1   &   8   &   268   &   26 \\
$\lambda = 10^{-4}$ &
3   &   117   &   45   &   50   &   18   &   14   &   1   &   1   &   16   &   262   &   31 \\
$\lambda = 10^{-2}$ &
3   &   119   &   36   &   44   &   29   &   7   &   1   &   1   &   4   &   241   &   26 \\
\midrule
$\lambda = 0$ &
4   &   64   &   54   &   42   &   27   &   9   &   1   &   1   &   3   &   201   &   24 \\
$\lambda = 10^{-4}$ &
4   &   62   &   75   &   34   &   17   &   12   &   1   &   1   &   5   &   207   &   22 \\
$\lambda = 10^{-2}$ &
4   &   73   &   45   &   39   &   16   &   15   &   1   &   1   &   4   &   194   &   22 \\
\midrule
$\lambda = 0$ &
5   &   104   &   57   &   37   &   22   &   11   &   1   &   1   &   7   &   240   &   34 \\
$\lambda = 10^{-4}$ &
5   &   132   &   50   &   37   &   26   &   7   &   1   &   1   &   10   &   264   &   34 \\
$\lambda = 10^{-2}$ &
5   &   111   &   60   &   57   &   30   &   17   &   1   &   1   &   6   &   283   &   41 \\
\midrule
$\lambda = 0$ &
6   &   131   &   78   &   42   &   23   &   19   &   1   &   1   &   8   &   303   &   43 \\
$\lambda = 10^{-4}$ &
6   &   145   &   73   &   43   &   23   &   19   &   1   &   1   &   7   &   312   &   43 \\
$\lambda = 10^{-2}$ &
6   &   194   &   73   &   39   &   23   &   12   &   1   &   1   &   6   &   349   &   64 \\
\bottomrule    
\end{tabular}
\end{table}
In order to also give a qualitative expression of the
produced controls, we visualize them in
\Cref{fig:final_controls,fig:delta_0_final_controls}
for $\delta = 0$, $\delta = 1$ and $\lambda = 10^{-2}$.
\begin{figure}[!h]
    \centering
    \begin{subfigure}{.33\textwidth}
    \begin{tikzpicture}
    \begin{axis}[width=\textwidth,xlabel=$t\quad {\tiny(w^0 \equiv -7)}$,ymin=-5.25,ymax=0.25]
    \addplot[const plot, no marks]
        table[x index=0, y index=1, col sep=comma]
        {./data/w_1_0_1.00e-02_ini.csv};
    \addplot[const plot, dashed, very thick, no marks]
        table[x index=0, y index=1, col sep=comma]
        {./data/w_1_0_1.00e-02.csv};        
    \end{axis}
    \end{tikzpicture}
    \end{subfigure}\hfill
    \begin{subfigure}{.33\textwidth}
    \begin{tikzpicture}
    \begin{axis}[width=\textwidth,xlabel=$t\quad {\tiny(w^0 \equiv -5)}$,ymin=-7.25,ymax=0.25]
    \addplot[const plot, no marks]
        table[x index=0, y index=1, col sep=comma]
        {./data/w_2_0_1.00e-02_ini.csv};
    \addplot[const plot, dashed, very thick, no marks]
        table[x index=0, y index=1, col sep=comma]
        {./data/w_2_0_1.00e-02.csv};        
    \end{axis}
    \end{tikzpicture}
    \end{subfigure}\hfill
    \begin{subfigure}{.33\textwidth}
    \begin{tikzpicture}
    \begin{axis}[width=\textwidth,xlabel=$t\enskip (w^0 \equiv -3)$,ymin=-7.25,ymax=0.25]
    \addplot[const plot, no marks]
        table[x index=0, y index=1, col sep=comma]
        {./data/w_3_0_1.00e-02_ini.csv};
    \addplot[const plot, dashed, very thick, no marks]
        table[x index=0, y index=1, col sep=comma]
        {./data/w_3_0_1.00e-02.csv};    
    \end{axis}
    \end{tikzpicture}
    \end{subfigure}\\
    \begin{subfigure}{.33\textwidth}
    \begin{tikzpicture}
    \begin{axis}[width=\textwidth,xlabel=$t\quad {\tiny(w^0 \equiv -1)}$,ymin=-7.25,ymax=0.25]
    \addplot[const plot, no marks]
        table[x index=0, y index=1, col sep=comma]
        {./data/w_4_0_1.00e-02_ini.csv};
    \addplot[const plot, dashed, very thick, no marks]
        table[x index=0, y index=1, col sep=comma]
        {./data/w_4_0_1.00e-02.csv};    
    \end{axis}
    \end{tikzpicture}
    \end{subfigure}\hfill
    \begin{subfigure}{.33\textwidth}
    \begin{tikzpicture}
    \begin{axis}[width=\textwidth,xlabel=$t\quad {\tiny(w^0 \equiv 0)}$,ymin=-7.25,ymax=0.25]
    \addplot[const plot, no marks]
        table[x index=0, y index=1, col sep=comma]
        {./data/w_5_0_1.00e-02_ini.csv};
    \addplot[const plot, dashed, very thick, no marks]
        table[x index=0, y index=1, col sep=comma]
        {./data/w_5_0_1.00e-02.csv};  
    \end{axis}
    \end{tikzpicture}
    \end{subfigure}\hfill
    \begin{subfigure}{.33\textwidth}
    \begin{tikzpicture}
    \begin{axis}[width=\textwidth,xlabel=$t\enskip (w^0 \equiv 2)$,ymin=-7.25,ymax=0.25]
    \addplot[const plot, no marks]
        table[x index=0, y index=1, col sep=comma]
        {./data/w_6_0_1.00e-02_ini.csv};
    \addplot[const plot, dashed, very thick, no marks]
        table[x index=0, y index=1, col sep=comma]
        {./data/w_6_0_1.00e-02.csv};  
    \end{axis}
    \end{tikzpicture}
    \end{subfigure}    
    \caption{Control function $\psi$, poroelastic case $\delta = 0$, $\lambda = 10^{-2}$: final control functions produced 
    for the unregularized optimization (solid)
    and for the homotopy (dashed).}
    \label{fig:delta_0_final_controls}
\end{figure}

\begin{figure}[!h]
    \centering
    \begin{subfigure}{.33\textwidth}
    \begin{tikzpicture}
    \begin{axis}[width=\textwidth,xlabel=$t\quad {\tiny(w^0 \equiv -7)}$,ymin=-3.25,ymax=0.25]
    \addplot[const plot, no marks]
        table[x index=0, y index=1, col sep=comma]
        {./data/w_1_1_1.00e-02_ini.csv};
    \addplot[const plot, dashed, very thick, no marks]
        table[x index=0, y index=1, col sep=comma]
        {./data/w_1_1_1.00e-02.csv};        
    \end{axis}
    \end{tikzpicture}
    \end{subfigure}\hfill
    \begin{subfigure}{.33\textwidth}
    \begin{tikzpicture}
    \begin{axis}[width=\textwidth,xlabel=$t\quad {\tiny(w^0 \equiv -5)}$,ymin=-3.25,ymax=0.25]
    \addplot[const plot, no marks]
        table[x index=0, y index=1, col sep=comma]
        {./data/w_2_1_1.00e-02_ini.csv};
    \addplot[const plot, dashed, very thick, no marks]
        table[x index=0, y index=1, col sep=comma]
        {./data/w_2_1_1.00e-02.csv};        
    \end{axis}
    \end{tikzpicture}
    \end{subfigure}\hfill
    \begin{subfigure}{.33\textwidth}
    \begin{tikzpicture}
    \begin{axis}[width=\textwidth,xlabel=$t\enskip (w^0 \equiv -3)$,ymin=-3.25,ymax=0.25]
    \addplot[const plot, no marks]
        table[x index=0, y index=1, col sep=comma]
        {./data/w_3_1_1.00e-02_ini.csv};
    \addplot[const plot, dashed, very thick, no marks]
        table[x index=0, y index=1, col sep=comma]
        {./data/w_3_1_1.00e-02.csv};    
    \end{axis}
    \end{tikzpicture}
    \end{subfigure}\\
    \begin{subfigure}{.33\textwidth}
    \begin{tikzpicture}
    \begin{axis}[width=\textwidth,xlabel=$t\quad {\tiny(w^0 \equiv -1)}$,ymin=-3.25,ymax=0.25]
    \addplot[const plot, no marks]
        table[x index=0, y index=1, col sep=comma]
        {./data/w_4_1_1.00e-02_ini.csv};
    \addplot[const plot, dashed, very thick, no marks]
        table[x index=0, y index=1, col sep=comma]
        {./data/w_4_1_1.00e-02.csv};    
    \end{axis}
    \end{tikzpicture}
    \end{subfigure}\hfill
    \begin{subfigure}{.33\textwidth}
    \begin{tikzpicture}
    \begin{axis}[width=\textwidth,xlabel=$t\quad {\tiny(w^0 \equiv 0)}$,ymin=-3.25,ymax=0.25]
    \addplot[const plot, no marks]
        table[x index=0, y index=1, col sep=comma]
        {./data/w_5_1_1.00e-02_ini.csv};
    \addplot[const plot, dashed, very thick, no marks]
        table[x index=0, y index=1, col sep=comma]
        {./data/w_5_1_1.00e-02.csv};  
    \end{axis}
    \end{tikzpicture}
    \end{subfigure}\hfill
    \begin{subfigure}{.33\textwidth}
    \begin{tikzpicture}
    \begin{axis}[width=\textwidth,xlabel=$t\enskip (w^0 \equiv 2)$,ymin=-4.25,ymax=0.25]
    \addplot[const plot, no marks]
        table[x index=0, y index=1, col sep=comma]
        {./data/w_6_1_1.00e-02_ini.csv};
    \addplot[const plot, dashed, very thick, no marks]
        table[x index=0, y index=1, col sep=comma]
        {./data/w_6_1_1.00e-02.csv};  
    \end{axis}
    \end{tikzpicture}
    \end{subfigure}    
    \caption{Control function $\psi$, poroviscoelastic case $\delta = 1$, $\lambda = 10^{-2}$: final control functions produced 
    for the unregularized optimization (solid) and for the homotopy (dashed).}
    \label{fig:final_controls}
\end{figure}
The unregularized optimization terminates
after taking
between 21 and 64 iterations for $\delta = 0$ and taking between 24 and 67 iterations for $\delta = 1$.
The homotopy takes a cumulative number of iterations
between 35 and 90 iterations for $\delta = 0$ and takes a much
higher cumulative number between 194 and 349 iterations for
$\delta = 1$.

The objective values with the unregularized optimization problem
are very similar to the objective values produced by the homotopy
for $\delta = 0$ with relative differences generally below $10^{-4}$. 
This is different for $\delta = 1$, where the homotopy generally 
achieves  lower objective values with relative differences generally 
around $10^{-3}$.

For $\delta = 0$ the remaining instationarities are generally
similar and of the same order of magnitude for the unregularized
optimization and the homotopy. For $\delta = 1$, the
final instationarities obtained with the unregularized optimization
are generally (but not in all cases) between one and two orders of 
magnitude higher.

\paragraph{Input choice $S$.}
We report the details in \Cref{tab:delta_0_first_experiment_S} for
the poroelastic case $\delta = 0$ and in \Cref{tab:first_experiment_S}
for the poroviscoelastic case. Detailed iteration numbers
over the different values of the homotopy are given in 
\Cref{tab:delta_0_iterations_S} for $\delta = 0$
and \Cref{tab:iterations_S} for $\delta = 1$.

\begin{table}[t]
    \centering
    \caption{Control function $S$, poroelastic case $\delta = 0$: Objective, instationarity, and (cumulative)
    outer iterations of \eqref{eq:p} for six initial controls achieved by
    executing \Cref{alg:slip} on the discretized problem as well
    as by executing a homotopy, abbreviated (H),
    of executions of \Cref{alg:slip}
    on regularized problems with regularization parameter driven to
    zero for the same initial controls.}
    \label{tab:delta_0_first_experiment_S}
    \begin{tabular}{lr|llllll}
    \toprule
    & Ini.
    & Final Obj.
    & Final Obj. (H)
    & Final Inst.
    & Final Inst. (H) 
    & Iter.
    & Cum.\ Iter.\ (H)
    \\
\midrule
$\lambda = 0$ &
1   &   \num{1.2420}   &   \num{1.2417}   &   \num{1.0e-05}   &   \num{9.1e-06}   &   16   &   24 \\
$\lambda = 10^{-4}$ &
1   &   \num{1.2420}   &   \num{1.2417}   &   \num{1.0e-05}   &   \num{9.2e-06}   &   16   &   24 \\
$\lambda = 10^{-2}$ &
1   &   \num{1.2454}   &   \num{1.2451}   &   \num{2.0e-05}   &   \num{1.9e-05}   &   16   &   24 \\
\midrule
$\lambda = 0$ &
2   &   \num{1.2423}   &   \num{1.2421}   &   \num{1.5e-05}   &   \num{1.2e-05}   &   25   &   33 \\
$\lambda = 10^{-4}$ &
2   &   \num{1.2423}   &   \num{1.2422}   &   \num{1.5e-05}   &   \num{1.2e-05}   &   25   &   33 \\
$\lambda = 10^{-2}$ &
2   &   \num{1.2602}   &   \num{1.2455}   &   \num{4.4e-07}   &   \num{2.4e-06}   &   29   &   40 \\
\midrule
$\lambda = 0$ &
3   &   \num{1.2428}   &   \num{1.2424}   &   \num{8.8e-06}   &   \num{1.2e-05}   &   24   &   32 \\
$\lambda = 10^{-4}$ &
3   &   \num{1.2429}   &   \num{1.2424}   &   \num{8.7e-06}   &   \num{1.3e-05}   &   26   &   32 \\
$\lambda = 10^{-2}$ &
3   &   \num{1.2462}   &   \num{1.2458}   &   \num{9.5e-07}   &   \num{2.2e-05}   &   26   &   32 \\
\midrule
$\lambda = 0$ &
4   &   \num{1.2417}   &   \num{1.2428}   &   \num{1.1e-05}   &   \num{1.5e-05}   &   15   &   27 \\
$\lambda = 10^{-4}$ &
4   &   \num{1.2417}   &   \num{1.2428}   &   \num{1.1e-05}   &   \num{1.5e-05}   &   15   &   27 \\
$\lambda = 10^{-2}$ &
4   &   \num{1.2450}   &   \num{1.2462}   &   \num{2.1e-05}   &   \num{2.5e-05}   &   15   &   29 \\
\midrule
$\lambda = 0$ &
5   &   \num{1.2414}   &   \num{1.2423}   &   \num{1.1e-05}   &   \num{1.6e-05}   &   12   &   23 \\
$\lambda = 10^{-4}$ &
5   &   \num{1.2414}   &   \num{1.2424}   &   \num{1.1e-05}   &   \num{1.6e-05}   &   12   &   23 \\
$\lambda = 10^{-2}$ &
5   &   \num{1.2448}   &   \num{1.2457}   &   \num{2.1e-05}   &   \num{6.3e-06}   &   12   &   24 \\
\midrule
$\lambda = 0$ &
6   &   \num{1.2553}   &   \num{1.2556}   &   \num{1.1e-05}   &   \num{2.5e-05}   &   16   &   28 \\
$\lambda = 10^{-4}$ &
6   &   \num{1.2554}   &   \num{1.2556}   &   \num{1.1e-05}   &   \num{2.5e-05}   &   16   &   28 \\
$\lambda = 10^{-2}$ &
6   &   \num{1.2602}   &   \num{1.2605}   &   \num{1.1e-05}   &   \num{1.0e-05}   &   13   &   29 \\
    \bottomrule    
    \end{tabular}
\end{table}
\begin{table}[t]
    \caption{Control function $S$, poroviscoelastic case $\delta = 1$: Objective,
    instationarity, and (cumulative) outer iterations of \eqref{eq:p} for six initial
    controls achieved by executing \Cref{alg:slip} on the discretized problem as well
    as by executing a homotopy, abbreviated (H), of executions of \Cref{alg:slip}
    on regularized problems with regularization parameter driven to
    zero for the same initial controls.}
    \label{tab:first_experiment_S}
    \begin{tabular}{lr|llllll}
    \toprule
    & Ini.
    & Final Obj.
    & Final Obj. (H)
    & Final Inst.
    & Final Inst. (H) 
    & Iter.
    & Cum.\ Iter.\ (H)
    \\
\midrule
$\lambda = 0$ &
1   &   \num{1.8162}   &   \num{1.4783}   &   \num{3.6e-03}   &   \num{5.2e-05}   &   10   &   45 \\
$\lambda = 10^{-4}$ &
1   &   \num{1.8162}   &   \num{1.4783}   &   \num{3.6e-03}   &   \num{5.2e-05}   &   10   &   45 \\
$\lambda = 10^{-2}$ &
1   &   \num{1.8185}   &   \num{1.4814}   &   \num{3.5e-03}   &   \num{4.2e-05}   &   10   &   42 \\
\midrule
$\lambda = 0$ &
2   &   \num{1.5849}   &   \num{1.4930}   &   \num{1.2e-03}   &   \num{2.3e-04}   &   20   &   57 \\
$\lambda = 10^{-4}$ &
2   &   \num{1.5849}   &   \num{1.4930}   &   \num{1.2e-03}   &   \num{2.3e-04}   &   20   &   57 \\
$\lambda = 10^{-2}$ &
2   &   \num{1.5755}   &   \num{1.4971}   &   \num{1.0e-03}   &   \num{2.2e-04}   &   21   &   58 \\
\midrule
$\lambda = 0$ &
3   &   \num{1.5480}   &   \num{1.5060}   &   \num{9.7e-04}   &   \num{1.5e-04}   &   17   &   53 \\
$\lambda = 10^{-4}$ &
3   &   \num{1.5481}   &   \num{1.5060}   &   \num{9.7e-04}   &   \num{1.5e-04}   &   17   &   53 \\
$\lambda = 10^{-2}$ &
3   &   \num{1.6598}   &   \num{1.5092}   &   \num{1.5e-03}   &   \num{1.4e-04}   &   17   &   58 \\
\midrule
$\lambda = 0$ &
4   &   \num{2.0874}   &   \num{1.5176}   &   \num{3.9e-03}   &   \num{3.1e-04}   &   15   &   49 \\
$\lambda = 10^{-4}$ &
4   &   \num{2.0874}   &   \num{1.5177}   &   \num{3.9e-03}   &   \num{3.1e-04}   &   15   &   49 \\
$\lambda = 10^{-2}$ &
4   &   \num{2.0279}   &   \num{1.4846}   &   \num{3.6e-03}   &   \num{7.3e-05}   &   18   &   50 \\
\midrule
$\lambda = 0$ &
5   &   \num{1.7936}   &   \num{1.4807}   &   \num{1.0e-03}   &   \num{3.6e-05}   &   8   &   46 \\
$\lambda = 10^{-4}$ &
5   &   \num{1.7936}   &   \num{1.4808}   &   \num{1.0e-03}   &   \num{3.6e-05}   &   8   &   46 \\
$\lambda = 10^{-2}$ &
5   &   \num{1.8013}   &   \num{1.4838}   &   \num{1.2e-03}   &   \num{2.6e-05}   &   10   &   49 \\
\midrule
$\lambda = 0$ &
6   &   \num{1.4730}   &   \num{1.4713}   &   \num{3.6e-04}   &   \num{2.8e-04}   &   14   &   45 \\
$\lambda = 10^{-4}$ &
6   &   \num{1.4731}   &   \num{1.4713}   &   \num{3.6e-04}   &   \num{2.8e-04}   &   14   &   45 \\
$\lambda = 10^{-2}$ &
6   &   \num{1.4776}   &   \num{1.4752}   &   \num{3.6e-04}   &   \num{2.2e-04}   &   14   &   46 \\
    \bottomrule    
    \end{tabular}
\end{table}

\begin{table}[!ht]
    \centering
    \caption{Control function $S$, poroelastic case $\delta = 0$: Number of iterations required by the executions of 
    \Cref{alg:slip} forthe different initializations of the computational example over the different
    values of $\varepsilon$ of the homotopy, cumulative for the homotopy (H), and for the unregularized problem (U).}
    \label{tab:delta_0_iterations_S}
    \begin{tabular}{lr|llllllllll}
    \toprule
    && $\varepsilon = $ \\
    & Ini.
    & ${0.016}$
    & ${0.008}$
    & ${0.004}$
    & ${0.002}$
    & ${0.001}$
    & ${0.0005}$
    & ${0.00025}$
    & ${0}$
    & (H)
    & (U)
    \\
\midrule
$\lambda = 0$ &
1   &   14   &   2   &   1   &   1   &   1   &   1   &   2   &   2   &   24   &   16 \\
$\lambda = 10^{-4}$ &
1   &   14   &   2   &   1   &   1   &   1   &   1   &   2   &   2   &   24   &   16 \\
$\lambda = 10^{-2}$ &
1   &   14   &   2   &   1   &   1   &   1   &   1   &   2   &   2   &   24   &   16 \\
\midrule
$\lambda = 0$ &
2   &   20   &   1   &   1   &   2   &   2   &   3   &   1   &   3   &   33   &   25 \\
$\lambda = 10^{-4}$ &
2   &   20   &   1   &   1   &   2   &   2   &   3   &   1   &   3   &   33   &   25 \\
$\lambda = 10^{-2}$ &
2   &   25   &   3   &   1   &   3   &   2   &   2   &   1   &   3   &   40   &   29 \\
\midrule
$\lambda = 0$ &
3   &   19   &   3   &   1   &   2   &   3   &   1   &   1   &   2   &   32   &   24 \\
$\lambda = 10^{-4}$ &
3   &   19   &   3   &   1   &   2   &   3   &   1   &   1   &   2   &   32   &   26 \\
$\lambda = 10^{-2}$ &
3   &   19   &   2   &   2   &   2   &   3   &   1   &   1   &   2   &   32   &   26 \\
\midrule
$\lambda = 0$ &
4   &   16   &   2   &   1   &   2   &   3   &   1   &   1   &   1   &   27   &   15 \\
$\lambda = 10^{-4}$ &
4   &   16   &   2   &   1   &   2   &   3   &   1   &   1   &   1   &   27   &   15 \\
$\lambda = 10^{-2}$ &
4   &   16   &   3   &   1   &   3   &   3   &   1   &   1   &   1   &   29   &   15 \\
\midrule
$\lambda = 0$ &
5   &   11   &   1   &   2   &   3   &   2   &   2   &   1   &   1   &   23   &   12 \\
$\lambda = 10^{-4}$ &
5   &   11   &   1   &   2   &   3   &   2   &   2   &   1   &   1   &   23   &   12 \\
$\lambda = 10^{-2}$ &
5   &   12   &   1   &   2   &   2   &   2   &   3   &   1   &   1   &   24   &   12 \\
\midrule
$\lambda = 0$ &
6   &   15   &   2   &   2   &   1   &   4   &   2   &   1   &   1   &   28   &   16 \\
$\lambda = 10^{-4}$ &
6   &   15   &   2   &   2   &   1   &   4   &   2   &   1   &   1   &   28   &   16 \\
$\lambda = 10^{-2}$ &
6   &   12   &   2   &   3   &   1   &   3   &   1   &   4   &   3   &   29   &   13 \\
\bottomrule    
\end{tabular}
\end{table}

\begin{table}[!ht]
    \centering
    \caption{Control function $S$, poroviscoelastic case $\delta = 1$: Number of iterations required by the executions
    of \Cref{alg:slip} for the different initializations of the computational example over the different
    values of $\varepsilon$ of the homotopy, cumulative for the homotopy (H), and for the unregularized problem (U).}
    \label{tab:iterations_S}
    \begin{tabular}{lr|llllllllll}
    \toprule
    && $\varepsilon = $ \\
    & Ini.
    & ${0.016}$
    & ${0.008}$
    & ${0.004}$
    & ${0.002}$
    & ${0.001}$
    & ${0.0005}$
    & ${0.00025}$
    & ${0}$
    & (H)
    & (U)
    \\
\midrule
$\lambda = 0$ &
1   &   21   &   4   &   9   &   2   &   1   &   3   &   3   &   2   &   45   &   10 \\
$\lambda = 10^{-4}$ &
1   &   21   &   4   &   9   &   2   &   1   &   3   &   3   &   2   &   45   &   10 \\
$\lambda = 10^{-2}$ &
1   &   20   &   10   &   2   &   1   &   1   &   3   &   3   &   2   &   42   &   10 \\
\midrule
$\lambda = 0$ &
2   &   36   &   4   &   4   &   3   &   3   &   3   &   2   &   2   &   57   &   20 \\
$\lambda = 10^{-4}$ &
2   &   36   &   4   &   4   &   3   &   3   &   3   &   2   &   2   &   57   &   20 \\
$\lambda = 10^{-2}$ &
2   &   38   &   4   &   3   &   3   &   3   &   3   &   2   &   2   &   58   &   21 \\
\midrule
$\lambda = 0$ &
3   &   25   &   16   &   2   &   2   &   2   &   1   &   3   &   2   &   53   &   17 \\
$\lambda = 10^{-4}$ &
3   &   25   &   16   &   2   &   2   &   2   &   1   &   3   &   2   &   53   &   17 \\
$\lambda = 10^{-2}$ &
3   &   29   &   17   &   2   &   2   &   2   &   1   &   3   &   2   &   58   &   17 \\
\midrule
$\lambda = 0$ &
4   &   31   &   4   &   3   &   3   &   1   &   2   &   3   &   2   &   49   &   15 \\
$\lambda = 10^{-4}$ &
4   &   31   &   4   &   3   &   3   &   1   &   2   &   3   &   2   &   49   &   15 \\
$\lambda = 10^{-2}$ &
4   &   30   &   4   &   3   &   3   &   1   &   2   &   3   &   4   &   50   &   18 \\
\midrule
$\lambda = 0$ &
5   &   31   &   2   &   1   &   3   &   3   &   2   &   2   &   2   &   46   &   8 \\
$\lambda = 10^{-4}$ &
5   &   31   &   2   &   1   &   3   &   3   &   2   &   2   &   2   &   46   &   8 \\
$\lambda = 10^{-2}$ &
5   &   31   &   2   &   1   &   3   &   3   &   3   &   4   &   2   &   49   &   10 \\
\midrule
$\lambda = 0$ &
6   &   26   &   3   &   2   &   2   &   5   &   3   &   2   &   2   &   45   &   14 \\
$\lambda = 10^{-4}$ &
6   &   26   &   3   &   2   &   2   &   5   &   3   &   2   &   2   &   45   &   14 \\
$\lambda = 10^{-2}$ &
6   &   27   &   3   &   2   &   2   &   5   &   3   &   2   &   2   &   46   &   14 \\
\bottomrule    
\end{tabular}
\end{table}
The unregularized optimization terminates after taking between 12 and 29 iterations for
$\delta = 0$ and taking between 8 and 21 iterations for $\delta = 1$. The homotopy takes
a cumulative number of iterations between 23 and 40 iterations for $\delta = 0$ and takes a much higher cumulative number
between 42 and 58 iterations for $\delta = 1$.

The objective values with the unregularized optimization problem have relative differences generally
around $10^{-2}$ compared to the objective values produced by the homotopy for $\delta = 0$.
This is different for $\delta = 1$, where the homotopy generally achieves lower objective values with
relative differences generally (but not in all cases) higher than $10^{-1}$.

For $\delta = 0$ the remaining instationarities are generally similar and of the same order of magnitude
for the unregularized optimization and the homotopy. For $\delta = 1$, the final instationarities obtained
with the unregularized optimization are generally (but not in all cases) between one and two orders of
magnitude higher.

\section{Conclusion}\label{sec:conclusion}
We investigated the regularity condition \Cref{ass:strong_assumption} that is required for \eqref{eq:p} for the 
convergence analysis of \Cref{alg:slip} in \cite{leyffer2022sequential} and proved a $\Gamma$-convergence 
result on a mollification of the control input as well as strict convergence for the iterates of a
corresponding homotopy trust-region algorithm.

We assess the proposed regularization for control problems governed by poroelastic and
poroviscoelastic equations modeling fluid flows through porous media. We considered the associated 1D
(in space) models with two possible control inputs (one acting in the interior and one acting on the boundary).
We showed that the regularity conditions are violated when the viscosity parameter $\delta$ is taken strictly
greater than zero or a Tikhonov term (for example, when $\lambda > 0$) is present. In comparison,
we proved that the necessary regularity conditions are satisfied for poroelastic systems (i.e., $\delta = 0$)
without Tikhonov term (i.e., $\lambda = 0$).

We applied \Cref{alg:slip} to instances of
\eqref{eq:p} for differently scaled 
Tikhonov terms, for both $\delta = 0$ and $\delta = 1$, and
for the two different control inputs that were analyzed
before. We observed that the presence of the Tikhonov term does
not seem to negatively impact the practical performance of
\Cref{alg:slip}, although \Cref{ass:strong_assumption} is always
violated in this case. If \Cref{ass:strong_assumption} is violated
due to the choice $\delta = 1$ (i.e., in the poroviscoelastic case), then the performance of \Cref{alg:slip}
and the quality of the final iterates it produces before
the trust region collapses is degraded.
This can be alleviated by executing a homotopy that drives
the support parameter of a mollification of the
control input into the PDE to zero over the course
of the optimization. However, the execution of the homotopy comes at
a higher computational cost. \clearpage%%%%%%%%%%%%%%%%%%%%%%%%%%%%%%%%%%%%%%%%%%%%%%%
\bibliographystyle{jnsao}
\bibliography{biblio}
%%%%%%%%%%%%%%%%%%%%%%%%%%%%%%%%%%%%%%%%%%%%%%%

%%%%%%%%%%%%%%%%%%%%%%%%%%%%%%%%%%%%%%%%%%%%%%%
\end{document}